%% file: hyper_rom.tex
\documentclass[a4paper,11pt]{article}
\usepackage[latin1]{inputenc}
\usepackage[T1]{fontenc} 
\usepackage{amsfonts}
\usepackage{amsmath,amsthm}
\usepackage{ stmaryrd }
\usepackage{graphicx}
\usepackage{amssymb}
\usepackage[scale=0.8]{geometry}
\usepackage{tikz}
\usepackage{tikz-cd}
\usepackage{color}
\usepackage{float}
\usepackage{subfigure}
 
\newtheorem{theorem}{Theorem}[section]
\newtheorem{algorithm}[theorem]{Algorithm}

\newtheorem{definition}[theorem]{Definition}

\newtheorem{lemma}[theorem]{Lemma}

\newtheorem{proposition}[theorem]{Proposition}

\newtheorem{remark}[theorem]{Remark}

\newcommand{\sL}{\mathrm{L}}

\newcommand{\RR}{\mathbb{R}}

\newcommand{\Pc}{\mathcal{P}}

\newcommand{\Vu}{ \boldsymbol{u}}

\newcommand{\Ve}{ \boldsymbol{e}}
 
\newcommand{\VF}{ \boldsymbol{F}}

\newcommand{\VPhi}{ \boldsymbol{\Phi}}
\newcommand{\Vphi}{ \boldsymbol{\phi}}

\begin{document}

\title{Reconstruction of finite volume solution for parameter-dependent linear hyperbolic conservation laws}

\author{M. Billaud-Friess and T. Heuz\'e}
\date{\today}
\maketitle

\begin{abstract}

This paper is concerned with the development of suitable  numerical method for the approximation of discontinuous solutions of parameter-dependent linear hyperbolic conservation laws.  The objective is to reconstruct such approximation, for new instances of the parameter values and given time, from a transformation of pre-computed snapshots of the solution for new  parameter values. 
In a finite volume setting,  a Reconstruct-Translate-Average (RTA) algorithm inspired from the Reconstruct-Evolve-Average one of Godunov's method is proposed. It allows to perform, in three steps, a transformation of the snapshots with piecewise constant reconstruction. The method is fully detailed and analyzed for solving a parameter-dependent transport equation for which the spatial transformation is related to the characteristic intrinsic to the problem.  Numerical results for transport equation and linear  elastodynamics equations illustrate the good behavior of the proposed approach.\\

\end{abstract}

{\small \noindent\textbf{Keywords: Parameter-dependent, Linear hyperbolic conservation laws, Reconstruct-Translate-Average, Finite Volume} 



\input{./introduction.tex}

\input{./finitevolume.tex}

\input{./reconstruction.tex}

\input{./transport.tex}
\input{./numericalexamples.tex}
\input{./conclusion.tex}


\nocite{*}
\bibliographystyle{plain} 
\bibliography{biblioHyp}

\end{document}

%% file: introduction.tex
\section{Introduction}\label{introduction}

\subsection{General context }\label{sec:soa}

Reduced Order Models (ROM) have emerged as a powerful tool to reduce the computational cost for solving complex numerical models whose solution belongs to a high-dimensional space. The main idea of these methods is to compute a surrogate model, whose solution is an approximation of the true solution, cheaper to compute. In practice, it allows to deal efficiently with real time problems requiring many evaluations of the solution as encountered in uncertainty quantification or parametric studies for example.  Many efforts have been devoted to the development of these methods for  elliptic and parabolic (parameter-dependent) Partial Differential Equations (PDEs). These problems usually possess smooth solutions which admit good approximations in low rank format. Projection based ROM methods are aimed at computing such a low rank approximation of the full solution in a low dimensional subspace  (see e.g. \cite{Nouy2015,Haasdonk2017}), called a reduced space, that well approximates the solution manifold of the full original problem.  This reduced space is generally spanned by a set of basis of vectors, computed from properly chosen snapshots of the full solution. 
For hyperbolic problems such methods provides poor approximations, in particular if the solution contains some discontinuities (see e.g. \cite{Haasdonk2009}).\\
Indeed, it is now well understood that such linear approximation methods are not suitable for solving hyperbolic problems because the solution manifold can not be well approximated with a linear space of small dimension $r$. Indeed, for such problems, the Kolmogorov $r$-width decreases slowly with $r$ (see e.g. \cite{Taddei2015,Welper2017,Greif2019}). To overcome the slow decay of  the  Kolmogorov $r$-width of the solution manifold, new ROM approaches have emerged relying on non linear approximation.
A first attempt for dealing with transport dominated problem has been considered in \cite{Ohlberger2013} using a freezing method. More recently, the authors of \cite{Rim2019} have proposed a projection based ROM method including non linear transformation of the solution manifold which produced adapted (local) reduced basis with respect to time and parameter. To do so, they introduce the concept of Manifold Approximation via Transported Subspaces (MATS), to obtain a time and parameter dependent reduced space from properly chosen transformations that belong to a {\it low rank space}  of dimension $s$. In that way, the authors generalize the notion of Kolmogorov $r$-width to non linear Kolmogorov $(r,s)$-width to the solution manifold which decreases faster with $r,s$. A similar idea has been considered for ROM of non linear transport-dominated problems in \cite{Peherstorfer2018}, to revisit the Empirical Interpolation Method (EIM) with adaptive basis and sampling. Other various approaches using transformed reduced basis have been proposed. Among them, let us mention the calibrated manifold based ROM \cite{Cagniart2016,Cagniart2017}, transformed snapshot interpolation method \cite{Welper2017}, the Shifted-POD (sPOD) \cite{Reiss2018,Black2020Nov} with application to multiple transport problems \cite{Black2019}, transport reversal for template fitting \cite{Rim2019}, Transported Snapshot ROM \cite{Sarna2020Mar}, the adaptive space-time registration-based data compression procedure \cite{Taddei2021Jan}. Beyond transformed based approaches, the resolution of parameter-dependent hyperbolic problems have motivated  the emergence of various ROM methods in the last years.  Let mention, for example, extension of ROM to $\sL^2$-Wasserstein spaces \cite{Ehrlacher2019} to solve one-dimensional conservation laws. In \cite{Bansal2021},  a ROM method relying on the decomposition of the solution into a function that tracks the evolving discontinuity and a residual part, combined with Proper Orthogonal Decomposition is proposed. Dynamical low-rank approximation methods seem also be relevant to deal with such problems, see \cite{Billaud2017} for parameter-dependent transport dominated problems and more recently \cite{Kusch2021May} for a Burger's equation with uncertainties. Finally, new path is opening up to methods gathering Neural Networks and MOR (see e.g. \cite{Rim2020Jul,Peng2021May}).\\
Morever, defining suitable projection based ROM methods providing approximations that preserve as much as possible the mathematical features of solutions of hyperbolic PDEs remain an opened question. Especially, the design of methods which are conservative, entropic, monotonic or Total Variation Diminishing (TVD). In this direction, it was proposed in \cite{Abgrall2016,Abgrall2018}  an approximation problem based on the minimization of the residual of the discretized equations in $\sL^1$-norm for ROM of hyperbolic conservation laws.  It has the advantage to provide non oscillatory approximation, especially in presence of shocks. Projection-based hyper-reduced models of nonlinear conservation laws  globally conservative that inherits a semi-discrete entropy inequality has also been proposed in \cite{Chan2020Dec}. 

\subsection{Main contribution and outline of the paper}


In this paper, we consider a one dimensional parameter-dependent scalar conservation equation. The space domain is the bounded open interval $\Omega \subset \RR$, and $I=[0,T]$ is the time domain. We are interested by the function $u(\cdot, t; \mu)$, belonging to the space $V$, solution of the equation
\begin{equation}\label{eq:syscons} 
\partial_tu(x,t;\mu) +  \partial_x f(u(x,t;\mu);\mu) =0, \quad (x,t) \in \Omega \times I,\\
\end{equation} 
with initial condition $u^0 : \overline \Omega \to \RR$ and suitable boundary conditions. 
The considered initial boundary value problem \eqref{eq:syscons} depends on parameters $\mu  \in \Pc \subset \RR^p$ through the conservative flux $f(\cdot,\mu)$.\\

In the lines of \cite{Peherstorfer2018,Rim2019,Welper2017},  one application of this work, is to  design a dynamical RB method \cite{Billaud2017} with adapted local basis for  solving Equation \eqref{eq:syscons}.
The idea of such an approach is to design reduced spaces spanned by local basis functions $\{ \phi_i(x,t; \mu) \}_{i=1}^r$ deduced from $\{u(\varphi_i(x,t;\mu),t), t; \mu_i)\}_{i=1}^r$, which correspond to transformed snapshots of the solution for given parameter instances with $\varphi(\cdot,t;\mu, \mu_i) : \Omega \to \Omega $ a parameter dependent space transformation at the instant $t$. Then, the  solution $u(x,t;\mu)$ of Equation\eqref{eq:syscons} is approximated by the following {\it rank-$r$} approximation
\begin{equation}\label{eq:approxt}
u(x,t;\mu) \approx \sum_{i=1}^r \alpha_i(t;\mu) \phi_i(x,t; \mu).
\end{equation}
From a physical viewpoint, a relevant choice is to derive the parameter dependent transformations $\varphi(\mu, \mu_i) $ from the characteristic associated with the hyperbolic system \eqref{eq:syscons}, such that the approximation \eqref{eq:approxt} captures well the features of the original solution. In particular, let us mention that such approximation is exact with  $r=1$ for parameter dependent linear transport equation when the space transformation is a parameter and time dependent space shift \cite[Example 3.6]{Rim2019}. \\

This paper focuses on the design of a robust approximation for discontinuous solutions generated by parameter-dependent linear hyperbolic systems. Efficiently computing the transformation $\varphi(\mu, \mu_i) $ is then essential, given precomputed snapshots of the trajectory of the true solution for given instances of the parameter. In particular, this work derives such an approach in full Finite Volume (FV) framework where the snapshots consist of trajectories of the numerical approximation provided by a known FV scheme. To this end, the Reconstruct-Translate-Average (RTA) method is introduced, which is inspired from Godunov's method also interpreted as Reconstruct-Evolve-Average (REA) \cite{Leveque2002}. The proposed method is detailed and analyzed for one dimensional parameter dependent linear scalar transport equation, and an application to system of linear hyperbolic conservation laws is shown. In that case, one snapshot is sufficient to approximate the FV solution for any parameter values. In view of ROM, the ideal approximation is recovered with one single basis function up to discretization error at any time, as for the continuous case.  This leads to a simple and efficient method that does not require any projection step neither time stepping procedure. \\


The outline of the paper is as follows.  After recalling in Section \ref{sec:FV} basic notations and results associated with FV scheme especially for transport equations, we present in Section \ref{sec:econs}  the RTA method in general setting and discuss possible application in the context of ROM. In Section \ref{sec:transport}, the RTA method is fully detailed for practical application, and convergence study is also performed. We conclude the paper with some numerical result demonstrating the behavior of the proposed method for the parameter dependent transport problem and the wave equation.

%% file: finitevolume.tex
\section{Finite volume scheme} \label{sec:FV}

In this section, basic notations of the finite volume approximation for parameter-dependent scalar conservation law are summarized in one space dimension. In particular, the first order upwind scheme is recalled for the transport equation, and its link with the so-called Reconstruct-Evolve-Average (REA) algorithm.

\subsection{General setting}

Let $\mathcal {T}_N  = \{\overline C_j\}_{j=1}^N$ be a uniform one-dimensional mesh of the spatial domain $\overline{\Omega}$  containing $N$ cells noted $C_j = (x_{j-1/2},x_{j+1/2})$ of size $\Delta x$ and node coordinates $x_{j\pm1/2} = (j\pm1/2)\Delta x$. The time interval is  discretized such that $0 = t^0 <  \dots < t^k < \dots <t^K=T$ with a fixed time-step $\Delta t$, and $t^k = k \Delta t$. In practice, $\Delta t$ is chosen so that to satisfy the Courant-Friedrich-Levy  (CFL) stability condition.\\


%

The conservation law \eqref{eq:syscons} written in integral form over the domain $C_j \times [t^k,t^{k+1}]$ yields the conservative time update
\begin{equation}
 \int_{C_j} u(x,t^{k+1};\mu) dx   =  \int_{C_j} u(x,t^k;\mu)  dx - \Delta t \left( F_{j+\frac{1}{2}}(\mu)- F_{j-\frac{1}{2}}(\mu)\right)
\label{eq:conservativeForm}
\end{equation} 
where $F_{j+\frac{1}{2}}(\mu)$ and $F_{j-\frac{1}{2}}(\mu)$ denotes the interface fluxes defined by
\begin{equation}
F_{j+\frac{1}{2}}(\mu) = \frac{1}{\Delta t} \int_{t^k}^{t^{k+1}} f(u(x_{j+\frac{1}{2}},t;\mu)) dt \:, \quad F_{j-\frac{1}{2}}(\mu) = \frac{1}{\Delta t} \int_{t^k}^{t^{k+1}} f(u(x_{j-\frac{1}{2}},t;\mu)) dt.
\end{equation}
The flux $F_{j - \frac{1}{2}}(\mu)$ (similarly for  $F_{j + \frac{1}{2}}(\mu)$),  at interface $j-1/2$ between cells $C_{j-1},C_{j}$, can then be approximated by numerical fluxes to derive a particular finite volume scheme. Especially, in the Godunov's method \cite{Godunov1959}, $F_{j-\frac{1}{2}}(\mu)$ is approximated by the Godunov's fluxes ${\cal F}(u^k_{j-1},u^k_j;\mu)$ where $u_j^k(\mu)$ stands for the approximated averaged valued of $u$ in the cell $C_j$ at time $t^k$.  It leads to a the first order upwind scheme that produces a sequence of approximations $\{u_N^k(\mu)\}_{k=0}^K$ in $V_N \subset V$ in the finite space $V_N$ of piecewise constant functions, with $\dim(V_N)=N$,  obtained with the following rule. Given the initial condition $u^0$, the  $\sL^2$-projection on $V_N$ of $u_N^0(\mu)$ given by
\begin{equation}
u_N^0(x; \mu) = \dfrac{1}{\Delta x} \int_{C_j} u^0, ~\text{for } x \in C_j,
\label{eq:projCI}
\end{equation}
then, $u_N^{k+1}(\mu) \in V_N$ is defined as
\begin{equation}
u_N^{k+1}(x; \mu) =  u_j ^{k+1}(\mu), ~\text{for }x \in C_j,
\end{equation}
where the cell values $u_j^{k+1}(\mu)$ are obtained through the time update
\begin{equation}
 \qquad u_j ^{k+1}(\mu) = u_j ^k(\mu) - \dfrac{\Delta t}{\Delta x} \left( {\cal F}(u_{j}^k,u_{j+1}^k;\mu)-  {\cal F} (u_{j-1}^k,u_{j}^k; \mu) \right).
\label{eq:conservativeForm}
\end{equation}

Equation \eqref{eq:conservativeForm} can also be written in vector form by identifying the sequence of approximation $\{u_N^k(\mu)\}_{k=0}^K$ in $V_N$ with the sequence of vectors $\{\Vu^k(\mu)\}_{k=0}^K$ in $\RR^N$ given by 
\begin{equation}
\Vu^{k+1}(\mu)=  \Vu^k(\mu) - \frac{\Delta t }{\Delta x}\left(\VF_+(\Vu^k;\mu)-  \VF_-(\Vu^k;\mu)\right),
\label{eq:FVformulation}
\end{equation}
with $\Vu^k(\mu) \in \RR^N$ having component $(\Vu^{k}(\mu))_j = u_j^k(\mu)$. The discrete flux vectors $\VF_\pm(\Vu^k;\mu) \in \RR^N$ are given by $(\VF_+(\Vu^k ;\mu))_j =  {\cal F}(u_{j}^k,u_{j+1}^k;\mu)$ and $(\VF_-(\Vu^k;\mu))_j =  {\cal F}(u_{j-1}^k,u_{j}^k;\mu)$. For the sake of readability, the dependence on $\mu$ in $u_j ^{k}$ is omitted when there is no ambiguity.

\subsection{Transport equation} \label{FVtransport}

A parameter dependent transport equation is considered with periodic boundary conditions in this section. The conservative flux in Equation \eqref{eq:syscons} is linear and simply reads $f(u,\mu) = a(\mu) u$, for a given uniformly bounded real valued function $a : \RR^p \to \RR$.
 In that case, the upwind numerical fluxes in \eqref{eq:FVformulation} are given by
\begin{eqnarray}
(\VF_{+}(\Vu^k ,\mu))_j &=& {\cal F}(u_{j+1}^k,u_j^k; \mu)=a_+(\mu) u_j^k +a_-(\mu)u_{j+1}^k,  \nonumber \\
(\VF_{-}(\Vu^k;\mu))_j &=& {\cal F}(u_{j}^k,u_{j-1}^k;\mu)=a_+(\mu) u_{j-1}^k+a_- (\mu)u_{j}^k, 
\label{eq:upwindfluxes}
\end{eqnarray}
where $a_+(\mu)=\max (a(\mu),0)$ and $a_-(\mu)=\min (a(\mu),0)$ are respectively the positive and negative parts of $a(\mu)$. A von Neumann analysis ensures that this scheme is stable under the CFL condition 
\begin{equation}
\label{eq:CFL}
|a(\mu)| \Delta t < \Delta x.
\end{equation}
The upwind scheme for the advection equation can be derived as a special case of the \textit{Reconstruct-Evolve-Average} (REA) algorithm  \cite[Section 4.10]{Leveque2002} originally proposed by Godunov \cite{Godunov1959} for Euler equations. It involves three steps summarized in Algorithm \ref{REA}.

\begin{algorithm}[REA algorithm] \label{REA}~\\
Given the initial condition $\Vu^0 \in \RR^N$, compute $\Vu^k$ for  $k\in \{1, \dots, K\}$ as follows.
\begin{enumerate}
\item\emph{Reconstruct}  the function $u_N : t \mapsto  u_N(t)$ from  $I$ to $V_N$ such that 
$$u_N(x,t^k) = u_j^k,  \quad  \text{ for } x \in C_j.$$
\item \emph{Evolve} the solution from $t^k$ to $t^{k+1}$, by computing the exact solution of the transport equation \eqref{eq:syscons} with the initial datum $u_N(t^k)$, to get 
$$
\tilde{u}_N(x,t^{k+1}) = u_N (x-a(\mu)\Delta t, t^k), \text{ for } x \in \Omega.
$$ 
\item \emph{Average}  the function $\tilde{u}_N(t^{k+1})$ over each grid cell $C_j$ to obtain
$$u_j^{k+1} = \frac{1}{\Delta x} \int_{C_j} \tilde{u}_N(x,t^{k+1})dx .$$
\end{enumerate}
\end{algorithm}

Under CFL condition \eqref{eq:CFL}, the REA algorithm, with piecewise constant reconstruction, is convergent \cite[\S 8]{Leveque2002}. Indeed, given $\mu \in \Pc$, for any time $t^k$ we have
\begin{equation}
 \lim_{\Delta x \to 0 } \|u_N(t^k; \mu) - u(t^k; \mu)\|_1 = 0 
\label{eq:upconv}
\end{equation}
where $\|\cdot\|_1$ stands for the $\sL^1$-norm in space. 
For smooth solutions, the error behaves as $\mathcal{O}(\Delta x)$, but only as $\mathcal{O}(\sqrt{\Delta x})$ for solutions including discontinuities \cite[\S 8]{Leveque2002}. The REA procedure applied with a piecewise constant reconstruction yields a monotone scheme \cite[Definition 13.35]{Toro2013}, preventing the appearence of new extrema in the numerical solution and hence of spurious numerical oscillations. Higher order approximation could also be recovered by using more general piecewise polynomial reconstruction at Step 2. (see e.g. \cite[\S 6]{Leveque2002} for details).\\


\subsection{Shifting operator}

The upwind scheme or equivalently the REA algorithm with piecewise constant reconstruction can be reformulated in term of a {\it shifting operator}. This operator originally derived in \cite[Section 3]{Rim2018Feb} is recalled in this section.\\

First we define the permutation matrix $L \in \RR^{n\times n}$, associated with periodic boundary conditions, by
$$L=
\left (
\begin{matrix}
0 & 0 & \cdots & 0 & 1 \\
1 & 0 & \cdots & 0 & 0 \\
0 & 1 & \cdots & 0 & 0 \\
\vdots & \vdots &  \ddots &\vdots & \vdots \\
0 & 0 & \cdots & 1 & 0 \\
\end{matrix}
\right ),
$$
which satisfies $LL^T =id_N$, $id_N$ being the $N$-order identity matrix. More generally, for two arbitrary integers $s,p$ the following property holds
$$
(L^s)^TL^p  = 
\left\{
\begin{array}{cl}
(L^{s-p})^T & \text{ if } s > p , \\
L^{p-s} & \text{ if } s < p , \\
id_N& \text{ if } p =s.
\end{array}
\right.
$$
Moreover, if $p$ is a negative integer, then  $L^p=  (L^{-p})^T$. 

%
%



\begin{definition}[Shifting operator] \label{def:gshift}
For any $\tilde \omega \in \RR $, we denote $N_{\tilde \omega} = \lfloor \omega \rfloor \in \mathbb{Z}$ its integer part and $\omega = \{ \tilde \omega\} \in [0,1)$ its fractional part. The shifting operator is the $N$-order matrix  defined as follows $${\cal K }(\tilde \omega) = L^{N_{\tilde \omega}}K(\omega). $$
with  $K(\omega) $ the $N$-order matrix  defined only for $\omega \in [0,1]$ by  $$K(\omega) = (1-\omega)I+\omega L.$$
For $\tilde \omega  \in [0,1]$, we have ${\cal K }(\tilde \omega)  = K(\omega)$.  
\end{definition}

Under the CFL condition \eqref{eq:CFL}, the upwind scheme given by Equations \eqref{eq:FVformulation}-\eqref{eq:upwindfluxes} can be recast under the following algebraic form 
\begin{equation}
 \qquad \Vu^{k+1}(\mu) = K(\nu) \Vu^{k}(\mu),
\label{eq:upwind_matrix_form}
\end{equation} 
with $\nu = a(\mu) \frac{\Delta t}{\Delta x}$.

%% file: reconstruction.tex
\section{Reconstruction of FV-solution}\label{sec:econs}

This section focuses on the computation of some function $\phi_i(\cdot,t ; \mu)$  for a target value of the parameter $\mu \in \Pc$  from given precomputed snapshot $u(\cdot, t;\mu_i), \mu_i \neq \mu$.  To this end, let $\varphi(\mu,\mu_i)$ be a parameter dependent transformation from $\Omega \times I$ to $\Omega$ given by
\begin{equation}
\label{eq:spaceshift}
\phi(\cdot,t ; \mu,\mu_i) = u(\cdot, t,\mu_i) \circ \varphi(\mu,\mu_i)(\cdot,t), \qquad t \in I. 
\end{equation}
In Section \ref{sec:RTA}, we transpose this idea to a FV framework with the {\it Recontruct-Transform-Average} method. This approach is a first step toward  ROM for discretized hyperbolic conservation laws as discussed in Section \ref{sec:torom}.
 
\subsection{Recontruct-Transform-Average algorithm}\label{sec:RTA}

Given a precomputed FV approximation $\{\Vu^k(\mu_i)\}_{k=0}^K$, satisfying Equation \eqref{eq:conservativeForm} for one instance $\mu_i \in \Pc$ of the parameter, we reconstruct the sequence of vectors $\{\Vphi^k_i(\mu)\}_{k=0}^K \subset \RR^N$ from this snapshot by means of the transformation $\varphi(\mu,\mu_i)$.  To mimic 
 the continuous transformation  $\varphi(\mu,\mu_i)$ in the discrete setting, the Reconstruct-Transform-Average (RTA) method is introduced here. In the lines of the REA method. At given time $t^k$, it consists of three steps. First, the global approximation $u_N^k(\mu_i)$ in $V_N$ at time $t^k$ is deduced from $\Vu^k(\mu_i)$. Second, the resulting function is transformed with $\varphi(\mu, \mu_i)(\cdot, t^k)$. Third,  the components  $\Vphi^k_{i,j}(\mu)$  of the vector $ \Vphi^k_i(\mu)$ are obtained  by averaging  the transformed snapshot in each grid cell  $C_j$. The overall RTA procedure is summarized in Algorithm \ref{alg:RTA}.

\begin{algorithm}[RTA algorithm] \label{alg:RTA}~\\[0.2cm]
Let be given $\mu, \mu_i \in \Pc$,  $\Vu^k(\mu_i)$ the snapshot associated with $\mu_i$ at time $t^k$, and the space transformation $\varphi(\mu,\mu_i)$. Compute $\Vphi_i^k (\mu)$ from the snapshot $\Vu^k(\mu_i)$ as follows.
\begin{enumerate}
\item \emph{Reconstruct} $u^k_N(\mu_i)\in V_N$ such that 
\begin{equation}
u^k_N(\mu_i)(x) =  u_j^k(\mu_i), \text{ for } x \in C_j. 
 \label{eq:reconstruct}
\end{equation}
\item \emph{Transform} the snapshot $u^k_N(\mu_i)$
\begin{equation}
\tilde  \phi^k_N(\mu,\mu_i)(x) =u^k_N(\mu_i) \circ \varphi(\mu,\mu_i)(x,t^k), \text{ for }  x \in \overline{\Omega}.
 \label{eq:transform}
\end{equation}
\item \emph{Average} $\tilde \phi^k_N(\mu)$ over each grid cell $C_j$ to obtain the components $\phi^k_{i,j}(\mu)$ of $\Vphi^k_i(\mu) \in \RR^N$ given by 
\begin{equation}
 \phi^k_{i,j}(\mu) = \dfrac{1}{\Delta x} \int_{C_j} \tilde \phi^k_N(\mu,\mu_i)(x) dx, \quad j = 1, \dots, N.
 \label{eq:average}
 \end{equation}
\end{enumerate}
\end{algorithm}

\begin{remark}
For a transport problem, when $\varphi(\mu,\mu_i)$ is chosen to be a space shift and the snapshots are computed with the upwind scheme given by Equation \eqref{eq:upwind_matrix_form},  the RTA method is proven to be Total Variation Bounded (TVB), see Section \ref{sec:stab}. The main avantage of the RTA algorithm in this case is that the approximation $\Vphi_i^k(\mu)$ reconstructed from the snapshot $\Vu^k_N(\mu)$  will not exhibit any spurious numerical oscillations.
\end{remark}

\begin{remark}
The extension of the RTA procedure to approximations of higher order is straighforward on cartesian grids. Moreover, it is not limited to any particular boundary condition or transformation form. Such extensions are the object of future works. 
\end{remark}

\subsection{Toward ROM} \label{sec:torom} 

As motivated in the introduction, one possible application of the RTA algorithm is nonlinear ROM for FV approximation of parameter dependent conservation laws. The objective is to properly approximate the solution manifold at each instant by means of a suitable parameter (and time) dependent reduced space.  The RTA algorithm can be seen as a first step for the construction of such adapted and parameter-dependent reduced basis, required in the approximation \eqref{eq:approxt}. At this point, it remains to define the transformation maps $\varphi(\mu,\mu_i)$. The particular case of space shift transformation in the line of \cite{Welper2017,Rim2018,Rim2018-2,Rim2018Feb,Reiss2018,Black2019,Black2020Nov,Sarna2020Mar} is discussed.  \\

The scalar transport equation with periodic boundary conditions considered in Section \ref{FVtransport}, admits the following exact solution
$$
u(x,t; \mu)=u_0 (x-a(\mu)t),
$$
which is a translation of the initial condition $u^0$ along the characteristic $x = x_0+a(\mu)t$ starting from $x_0$.
Given some snapshot of the solution for an instance of the parameter $\mu_i \in \Pc$ at time $t$, the target solution for the value $\mu \in\Pc$ reads
\begin{equation}
\label{eq:reecriture}
u(x,t; \mu)=u(\varphi(\mu,\mu_i)(x,t), t; \mu_i),
\end{equation}
with the transformation map $\varphi(\mu,\mu_i)= x -  (a(\mu)-a(\mu_i)) t$,  corresponding to a shift in the space domain. For such a choice, the approximation provided by Equation \eqref{eq:approxt} is exact with only a rank $r$ equal to $1$.  Transposing this observation to the FV discrete setting, the RTA algorithm can be expected to provide, up to some discretization error depending on $\Delta x$, an approximation  $\Vphi_i^k(\mu)$ of the true numerical solution $\Vu^k(\mu)$  reconstructed from the snapshot  $\Vu^k(\mu_i)$ \cite[Example 3.6]{Rim2019}. The corresponding procedure with online-offline implementation is discussed in the forthcoming section. In practice, it provides a very simple strategy, and allows to efficiently compute  approximations of FV solutions of parameter dependent linear conservation laws from pre-computed FV solutions.
For nonlinear conservation laws, the characteristics are generally complex to compute, or even non explicitly known. In that case, efficient ROM strategies computing at the same time both reduced approximation of the solution and suitable approximation of the transformation should be preferred \cite{Cagniart2017,Rim2019}.

%% file: transport.tex
\section{RTA for parameter-dependent transport equation}  \label{sec:transport}

In this section, the RTA method is detailed for the parameter-dependent transport equation. The boundedness and the convergence of the approximation are then shown. Then, an efficient online/offline procedure is discussed.

\subsection{Matrix form}

For a transport equation, the transformation map related to the characteristic (see Section \ref{sec:torom}) reads
\begin{equation}
\varphi(\mu,\mu_i)(x,t) = x - \Delta a(\mu,\mu_i) t.
\label{eq:transf_transp}
\end{equation}
At time $t^k$, it corresponds to a space shift  between the characteristics associated with the two values $\mu$  and $\mu_i$ of the parameter (see illustration (a) in Figure \ref{schema:sketch})  where $\Delta a(\mu,\mu_i) t^k = (a(\mu) - a(\mu_i))t^k = k(\nu-\nu_i) \Delta x$.

\begin{figure}[H]
\centering
\includegraphics[scale=1.2]{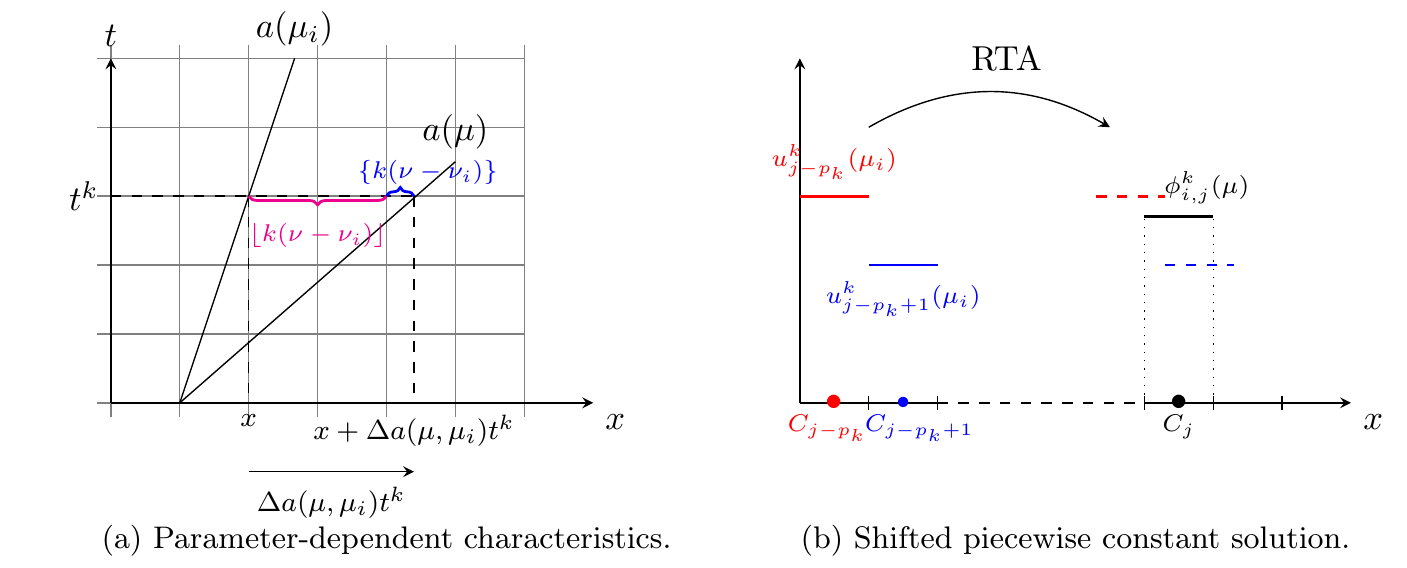}
\caption{Space shift with respect to $\mu$ at time $t^k$, illustrated for $\Delta a(\mu,\mu_i)> 0$.}
\label{schema:sketch}
\end{figure}

For any instant $t^k$,  the snapshot $\Vu^k(\mu_i)$   has been computed with the upwind scheme \eqref{eq:FVformulation}-\eqref{eq:upwindfluxes}.  Step 3 of Algorithm \ref{alg:RTA} consists in averaging the translated snapshots resulting from  Step 2 over the grid cell $C_j$ to get $\phi_{i,j}^k(\mu)$. For the sake of presentation, a uniform mesh is considered, each cell $C_j$ has a constant length $\Delta x$, although the procedure can easily be extended to non-uniform meshes. \\

The shifted snapshot is given in the cell $C_j$ by
\begin{equation}
\label{eq:shiftedsnap}
\tilde \phi_N^k(\mu,\mu_i)(x) = 
\left\{\begin{array}{rl}
u_{j-p_k}^k(\mu_i) & \text{if } x  < x^*, \\
u_{j-p_k+1}^k(\mu_i) & \text{if } x > x^*
\end{array}
\right.
\end{equation}
with $x^* \in C_j$ defined as 
\begin{equation}
\label{eq:xstar}
x^* = x_{j-p_k+1/2} + \Delta x k (\nu- \nu_i).
\end{equation}
The average value $u_{j-p_k}^k$ ($u^k_{j-p_k+1}$) associated with grid cell $C_{j-p_k}$ ($C_{j-p_k+1}$ respectively), is translated into cell $C_j$, see the sketch (b) in Figure \ref{schema:sketch}. Since periodic boundary conditions are considered, the parameter dependent index $p_k:=p_k(\mu,\mu_i) $ is defined modulo $N$ by
\begin{equation}\label{defp}
p_k(\mu,\mu_i) \equiv \left \lfloor k (\nu-\nu_i) \right \rfloor +1 ~ (\text{mod } N),
\end{equation}
where $\lfloor \cdot \rfloor$ denotes the floor function.  For the sake of readability, the dependence on  $(\mu,\mu_i)$ of the index $p_k$ is omitted when there is no ambiguity. The integral of the translated snapshot over $C_j$ (Equation \eqref{eq:average}) gives
\begin{equation}
\int_{C_j} \tilde \phi^k_N (\mu,\mu_i)(x)dx  = \int_{x_{j- 1/2}}^{x^*} u_{j-p_k}^k (\mu_i)dx 
+  \int_{x^*}^{x_{j+1/2}} u^k_{j-p_k+1}(\mu_i)dx.
\label{eq:Ij}
\end{equation}
The definition of $p_k$ leads to $(x_{j+1/2}-x^*) = \Delta x (1-\theta_i^k)$ and $(x^*-x_{j-1/2}) = \Delta x \theta_i^k$ with $\theta_i^k = \{k(\nu-\nu_i)\}$. 
Then, from Equation \eqref{eq:Ij} the following recurrence formula is obtained for the basis functions $\Vphi_i^k(\mu)$
\begin{equation}
\label{recurrence}
{\phi}_{i,j}^k(\mu) =  \left( 1- \theta_i^k \right) u_{j-p_k+1}^k(\mu_i) +  \theta_i^k u^k_{j-p_k} (\mu_i),  \quad j = 1, \dots, N.
\end{equation}
 At each time step $t^k$, the recurrence formula \eqref{recurrence}  is equivalent to   
\begin{equation}
\Vphi_i^k(\mu)= {\cal K}(k (\nu-\nu_i))  \Vu^k(\mu_i), \quad i=1, \ldots,r
 \label{recurrencem}
\end{equation}
with the generalized shifting operator ${\cal K}(k (\nu-\nu_i)) \in \RR^{N \times N}$ given by 
\begin{equation}
\label{eq:RTAshiftop}
 {\cal K}(k (\nu-\nu_i)) =
  K(\{k(\nu-\nu_i\}) L^{p_k-1}.
  \end{equation}
Let us remark that the formulas \eqref{recurrencem} and \eqref{eq:RTAshiftop} are also valid for  $\Delta a (\mu,\mu_i) <0$, in that case $p_k-1 \le 0$ so that $L^{p_k-1}= (L^{-(p_k-1)})^T$.

\begin{remark}\label{CFLlike}
Assume that $\Delta a (\mu,\mu_i)  t^k \le \Delta x$, then $k(\nu-\nu_i) \in [0,1] $ and $p=1$. In that case, Equation \eqref{recurrencem} becomes
$$
{\Vphi}_{i}^k(\mu)  =  K(k(\nu-\nu_i)) \Vu^k(\mu_i).
$$
Despite  the inequality $\Delta a (\mu,\mu_i)  t^k \le \Delta x$ can be interpreted as a CFL-like condition, it is not required here for the stability of the method. Indeed, equation \eqref{recurrencem}  is not associated with an explicit time integration scheme for the transport equation but rather pertains to a translation of a discrete solution at fixed time $t^k$. 
\end{remark}

\input{./error.tex}

\subsection{Practical aspects}

For each time step $k$, we recall that the vector $\Vphi_i^k(\mu)$ is given by
 $$
\Vphi_i^k(\mu) = {\cal K}(\theta_i^k)  \Vu^k(\mu_i)   = K(\theta_i^k)  L^{p_k-1} \Vu^k(\mu_i)   = 		 (1-\theta_i^k)  L^{p_k-1}  \Vu^k(\mu_i) +   \theta_i^k   L^{p_k}  \Vu^k(\mu_i).
$$
By multiplying this equation by $(L^{p_k-1})^T $ and taking the scalar product of the last equality with the canonical basis vector $\Ve_{j+p_k-1}$, we get 
\begin{equation}
\phi^k_{i,j+p_k-1} (\mu) =  (1-\theta_i^k)  u_j^k(\mu_i)  + \theta_i^ku_{j-1}^k(\mu_i).
\label{eq:online}
\end{equation}
In such a manner,  the component  $j+p_k-1$ of $\Vphi^k_{i}$ is a linear combinaison of the components  of $\Vu^k(\mu_i)$ and $L\Vu^k(\mu_i)$ depending on $\nu-\nu_i$ through $\theta_i^k = \{k(\nu-\nu_i)\}$.
In  the same lines as  projection based ROM methods, the proposed RTA procedure could be recast in an {\it offline-online} fashion.
First, during the online stage, we compute the snapshot  $\{\Vu^k(\mu_i)\}_{k=0}^K$ of the numerical solution for a given value of $\mu_i \in \Pc$. 
This step can be costly as it requires $K$ steps whose complexity depends on $N$ through matrix vector products required for the upwind scheme, see Equation \eqref{eq:upwind_matrix_form}.
During the online stage, for a given parameter $\mu$ and time $t^k$,  we evaluate the parameter dependent quantities,  namely $\theta_i^k$ and  $p_k$. Then, if required, the vector $\Vphi_i^k(\mu)$ is deduced from Equation \eqref{eq:online}. \\

Such an offline-online procedure allows to compute efficiently a reduced approximation of the FV solution $\Vu^k(\mu)$ of parameter-dependent  transport equation  from any snapshot  $\Vu^k(\mu_i)$ at each time $t^k$ for any parameter value $\mu$.  Indeed by computing $\Vphi^k_i(\mu)$ for any time $t^k$, we obtain without any time stepping procedure nor projection step a rank one approximation for time $t^k$ of $\Vu^k(\mu)$.\\

\begin{remark} \label{rk:dico}
Let  ${\cal D}_s^k = \{\Vu^k(\mu_1), \dots,  \Vu^k(\mu_s)\}$ be a dictionary formed from $s$ snapshots of the FV solution provided by Equation \eqref{eq:conservativeForm} at the time $t^k$. We could imagine a selective procedure to compute for a given time $t^k$ and a new instance $\mu$ of the parameter the best approximation $\Vphi_i^k(\mu)$ from a snapshot computed for a parameter $\mu_i$. The idea is to select in the dictionary, the snapshot which minimizes the approximation error $e(t^k,\mu,\mu_i)$. This point is under investigation, in particular  it requires a computable and sharp a posteriori error estimate bound for  $e(t^k,\mu,\mu_i)$.
  \end{remark}

%% file: error.tex
\subsection{Properties of the RTA  method}\label{sec:error}

The approximation properties of the RTA method are discussed for the reconstruction of FV solution of parameter dependent transport equation from snapshots. After proving that the method is Total-Variation Bounded  (TVB), we provide some estimate of the approximation error between a target solution $\Vu^k(\mu)$ and the reconstructed approximation $\Vphi_i^k(\mu)$ from a snapshot $\Vu^k(\mu_i)$  at time $t^k$. 

\subsubsection{Total Variation Bounded (TVB)} \label{sec:stab}


Let us first recall the definition \cite[Definition 13.59]{Toro2013} of the {\it total variation}  of $\Vu^k(\mu) \in \RR^N$,  the FV approximation at time $t^k$.  It  is  defined as
$$
TV(\Vu^k(\mu))  = \sum_{j=1}^N |u_{j+1}^k(\mu)- u_{j}^k(\mu) |.
$$
Similar definition holds for the total variation of the reconstructed approximation $\Vphi_i^k(\mu) \in \RR^N$. 

\begin{proposition}
Let $\{\Vu^k(\mu_i)\}_{k=0}^K$ be the FV solution provided by the upwind scheme  \eqref{eq:upwind_matrix_form} at time $t^k$ under the CFL condition \eqref{eq:CFL}, with $ TV(\Vu^0) < \infty $. The RTA method applied with a space shift transform $\varphi(\mu,\mu_i)$ given by Equation  \eqref{eq:spaceshift}, is TVB. Hence, the reconstructed approximation $\{\Vphi_i^k\}_{k=0}^K$ satisfies for all $k\ge 0$
$$
TV(\Vphi_i^k(\mu)) \le TV(\Vu^0).
$$
\end{proposition}

\begin{proof}
Using the recurrence formula \eqref{recurrence}, the total variation reads
$$
\begin{array}{rcl}
TV(\Vphi_i^k(\mu)) & = & \displaystyle \sum_{j=1}^N |(1-\theta_i^k)u_{j-p_k+1}^k+\theta_i^ku_{j-p_k}^k - (1-\theta_i^k)u_{j-p_k}^k-\theta_i^ku_{j-p_k-1}^k |, \\
 & \le & \displaystyle  (1-\theta_i^k) \sum_{j=1}^N |u_{j-p_k+1}^k-u^k_{j-p_k}| +\theta_i^k  \sum_{j=1}^N  | u_{j-p_k}^k - u_{j-p_k-1}^k |  = TV(\Vu^k)
\end{array}
$$
as we consider periodic boundary conditions. Then the result holds since the upwind scheme is also Total Variation Diminishing (TVD), and as  $ TV(\Vu^0)$  is finite.
\end{proof}

 \begin{remark}
More generally, the RTA method remains TVB as soon as  it is applied to reconstruct $\Vphi_i^k(\mu)$ from snapshots  $\Vu^k(\mu_i)$ provided by a TVD scheme (for more general conservation laws with periodic boundary conditions)  by means of space shift. 
\end{remark}

\subsubsection{Convergence}

Let us denote
$$
e(t^k, \mu, \mu_i) = \|u_N(t^k; \mu) - \phi_N(t^k; \mu, \mu_i)\|_1,
$$ 
the absolute error in $\sL^1$-norm between the true FV approximation $u_N(t^k; \mu) \in V_N$,
and the reconstructed approximation $\phi_N(t^k; \mu, \mu_i) \in V_N$ obtained through the RTA method associated with $\Vphi_i^k(\mu)$ at time $t^k$. This error can be bounded by three error contributions 
$$
\begin{array}{rcl} 
e(t^k, \mu, \mu_i) & = & 
\|u_N(\cdot,t^k; \mu) - u(\cdot,t^k; \mu) +  u(\cdot,t^k; \mu)- \phi_N(\cdot, t^k; \mu, \mu_i) \|_1\\
& \le &  \|u_N(\cdot,t^k; \mu) - u(\cdot,t^k; \mu)\|_1 + \| u(\varphi(\mu, \mu_i)(\cdot,t^k),t^k; \mu_i)-\phi_N(\cdot, t^k; \mu, \mu_i) \|_1\\
&\le &  \|u_N(\cdot,t^k; \mu) - u(\cdot,t^k; \mu)\|_1 +  \| u(\varphi(\mu, \mu_i)(\cdot,t^k),t^k; \mu_i)- 
 u_N(\varphi(\mu, \mu_i)(\cdot,t^k),t^k; \mu_i)\|_1\\
 &+& \|u_N(\varphi(\mu, \mu_i)(\cdot,t^k),t^k; \mu_i)- \phi_N(\cdot, t^k; \mu, \mu_i) \|_1
\end{array}
$$
as $u(x,t^k; \mu)=u(\varphi(\mu, \mu_i)(x,t^k),t^k; \mu_i)$ for the transport problem. Up to a change of variable in space, the error reads
\begin{equation}
\label{eq:erreur_approximation}
e(t^k, \mu, \mu_i) \le \|u_N(\cdot, t^k; \mu) - u(\cdot,t^k; \mu)\| +  \| u(\cdot,t^k; \mu_i)- 
 u_N(\cdot,t^k; \mu_i)\|_1 + \|(I-P_{V_N}) \tilde \phi_N^k(\mu,\mu_i) \|_1
\end{equation}
where $ \phi_N(\cdot ,t^k; \mu, \mu_i) = P_{V_N}(u_N(\varphi(\mu, \mu_i)(\cdot,t^k),t^k; \mu_i))$ with $P_{V_N}$ the $L^2$-projection on $V_N$.  The two first terms are related to the FV scheme error (see Section \ref{FVtransport}) computed with the values $\mu$ and $\mu_i$ of the parameter respectively. The last term  represents the $\sL^2$ projection error in $V_N$  for the shifted function $u_N(\varphi(\mu, \mu_i)(\cdot,t^k),t^k; \mu_i) := \tilde \phi_N^k(\mu,\mu_i)$ (see Step 3. of Algorithm \ref{alg:RTA}). In what follows, an estimate of the projection error is given. 

\begin{lemma}
Let $\{\Vu^k(\mu_i)\}_{k=0}^K$ be the FV solution provided by the upwind scheme  \eqref{eq:upwind_matrix_form} at time $t^k$ under the CFL condition \eqref{eq:CFL} with $ TV(\Vu^0) < \infty $. Let $\tilde \phi_N^k(\mu,\mu_i)$ be the shifted snapshot provided at Step 2 of Algorithm \ref{alg:RTA} with a space shift transform $\varphi(\mu,\mu_i)$ given by Equation  \eqref{eq:spaceshift}. The projection error  in $\sL^1$-norm satisfies 
$$
 \|(I-P_{V_N}) \tilde \phi_N^k(\mu,\mu_i) \|_1 = {\cal O}(\Delta x), $$
where the constant for $ {\cal O}(\Delta x)$ only depends on $TV(\Vu^0)$.
\label{lem:projection}
\end{lemma}

\begin{proof}
Let $e_{V_N}^k(\mu,\mu_i)= \tilde \phi_N^k(\mu,\mu_i)  - P_{V_N} \tilde \phi_N^k(\mu,\mu_i)$.
Combining Equations \eqref{eq:shiftedsnap}-\eqref{recurrence}, we obtain for $x \in C_j$
  $$
 e_{V_N}^k(\mu,\mu_i)(x)= 
  \llbracket u^k(\mu_i)    \rrbracket_{j-p_k+1/2}  \left\{\begin{array}{rl}
\theta_i^k-1 & \text{if } x< x^*, \\
\theta_i^k & \text{if } x> x^*
  \end{array}
  \right.
  $$
where $  \llbracket u^k(\mu_i)    \rrbracket_{j-p_k+1/2} = u_{j-p_k+1}^k(\mu_i) -u_{j-p_k}^k(\mu_i) $  is the jump of $u_N^k(\mu_i)$ at coordinate $x_{j-p_k+1/2}$. If $\{k(\nu-\nu_i)\} =0$, the shifted function coincides with its projection so that $e_{V_N}^k(\mu,\mu_i)= 0$.\\

Taking the $\sL^1$-norm over the cell $C_j$ leads to
$$
\|e_{V_N}^k(\mu,\mu_i)\|_{1,C_j} = \int_{C_j}|e_{V_N}^k(\mu,\mu_i)(x)| dx = |\llbracket u^k(\mu_i)\rrbracket_{j-p_k+1/2}| \left( (1-\theta_i^k) (x^*-x_{j-1/2})+\theta_i^k (x_{j+1/2}-x^*)\right).
$$
using  $(x_{j+1/2}-x^*) = \Delta x (1-\theta_i^k)$ and $(x^*-x_{j-1/2}) = \Delta x \theta_i^k$  together with $\theta_i^k \in [0,1]$,  we get 
$$
\|e_{V_N}^k(\mu,\mu_i)\|_{1,C_j}  = 2 \Delta x\left|u^k_{j-p_k+1}(\mu_i) - u^k_{j-p_k}(\mu_i)\right|(1-\theta_i^k)\theta_i^k.
$$
Summing over $j$, and up to index permutation, we get the global error estimate
$$
\|e_{V_N}^k(\mu,\mu_i)\|_{1,\Omega} = \sum_{j=1}^n \|e_{V_N}^k(\mu,\mu_i)\|_{1,C_j}= 2\Delta x  (1-\theta_i^k)\theta_i^k TV(\Vu^k(\mu_i)).
$$
Observing that $(1-\theta_i^k)\theta_i^k\le \frac 14$, and that the upwind scheme is TVD, we get 
$$
\|e_{V_N}^k(\mu,\mu_i)\|_{1,\Omega}  \le \frac {\Delta x}{2}  TV(\Vu^0). 
$$
\end{proof}

\begin{remark}
In our study, we express the error in term of the $\sL^1$-norm as it is natural for hyperbolic conservation laws. For linear hyperbolic problem, as the transport equation considered here,  one can show that the projection error in $\sL^2$-norm behaves as $ {\cal O} (\Delta x ^{1/2})$.
\end{remark}

We conclude this section, proving that the RTA method is convergent. 

\begin{proposition}
We consider $u_N(t^k; \mu) \in V_N$ the FV solution provided by the upwind scheme  \eqref{eq:upwind_matrix_form} at time $t^k$ under the CFL condition \eqref{eq:CFL} with $ TV(\Vu^0) < \infty $. Moreover, let $\phi(t^k; \mu, \mu_i) \in V_N$ be the reconstructed approximation obtained from $\Vphi_i^k(\mu)$ computed with  the RTA method with a space shift transform $\varphi(\mu,\mu_i)$ given by Equation  \eqref{eq:spaceshift}. We have that 
$$
 \lim_{\Delta x \to 0 } e(t^k, \mu, \mu_i) =0.
 $$
\end{proposition}

\begin{proof}
According to Equation \eqref{eq:erreur_approximation}, the approximation error is bounded by three contributions. The two first terms are related to the FV scheme error related to the parameters $\mu, \mu_i$. Under the CFL condition, these two terms vanish as $\Delta x \to 0$ since the upwind scheme is convergent by Equation \eqref{eq:upconv}. Moreover by Lemma \ref{lem:projection}, the last contribution behaves as ${\cal O}(\Delta x)$. So $ \lim_{\Delta x \to 0 } e(t^k, \mu, \mu_i) =0$ which proves the result.
\end{proof}

%% file: numericalexamples.tex
\section{Numerical examples}

In this part, the behavior of the proposed reconstruction method is illustrated for the approximation of the solutions of a parameter dependent transport problem and the wave equation.

\subsection{Scalar transport equation}

Consider a one-dimensional domain $\Omega = (-L,L)$, with $L=10$m, in which the scalar transport equation \eqref{eq:syscons} with a linear flux function holds, supplemented with periodic boundary conditions and the following initial condition
$$
u^0(x)=\left\{\begin{split}
& u_L \quad \text{if } - \frac{L}{3}\leq x \leq \frac{L}{3}, \\
& u_H \quad \text{otherwise,}
\end{split}\right. 
$$
where $u_L = -1$ and $u_H = 1$ denote initial low and high values. The wavespeed $a(\mu)$ depends on the parameter $\mu \in [0,1]$ through the affine function $a(\mu)=\alpha \mu + \beta$. Here $\alpha, \beta$ are two constants whose values are here set at $5$ and $2$ respectively. \\

In what follows, we perform some comparison between the reconstructed approximation $\Vphi_i^k(\mu)$ provided by the RTA method, from given snapshots $\Vu^k(\mu_i)$ computed with the upwind scheme, to approximate the FV solution $\Vu^k(\mu)$ for a target value of the parameter $\mu \in \Pc$. The mesh consists here of $N=250$ cells, and the maximum CFL number is set at $0.8$ for $\mu=1$. The RTA solution is here reconstructed from a snapshot obtained for $\mu_i=0.4$. Figure \ref{transportSolution} shows a comparison between RTA and FV approximation computed for a parameter value $\mu=0.8$ at different times.  It can be observed  that the two approximations coincide. Moreover, the RTA solution does not exhibit any oscillating behavior, nor tends to depart more from the FV reference solution with time.
\begin{figure}[H]
\subfigure[$t=7.22 \times 10^{-1}$ s.]{\label{transport1}
\includegraphics[width=0.33\textwidth]{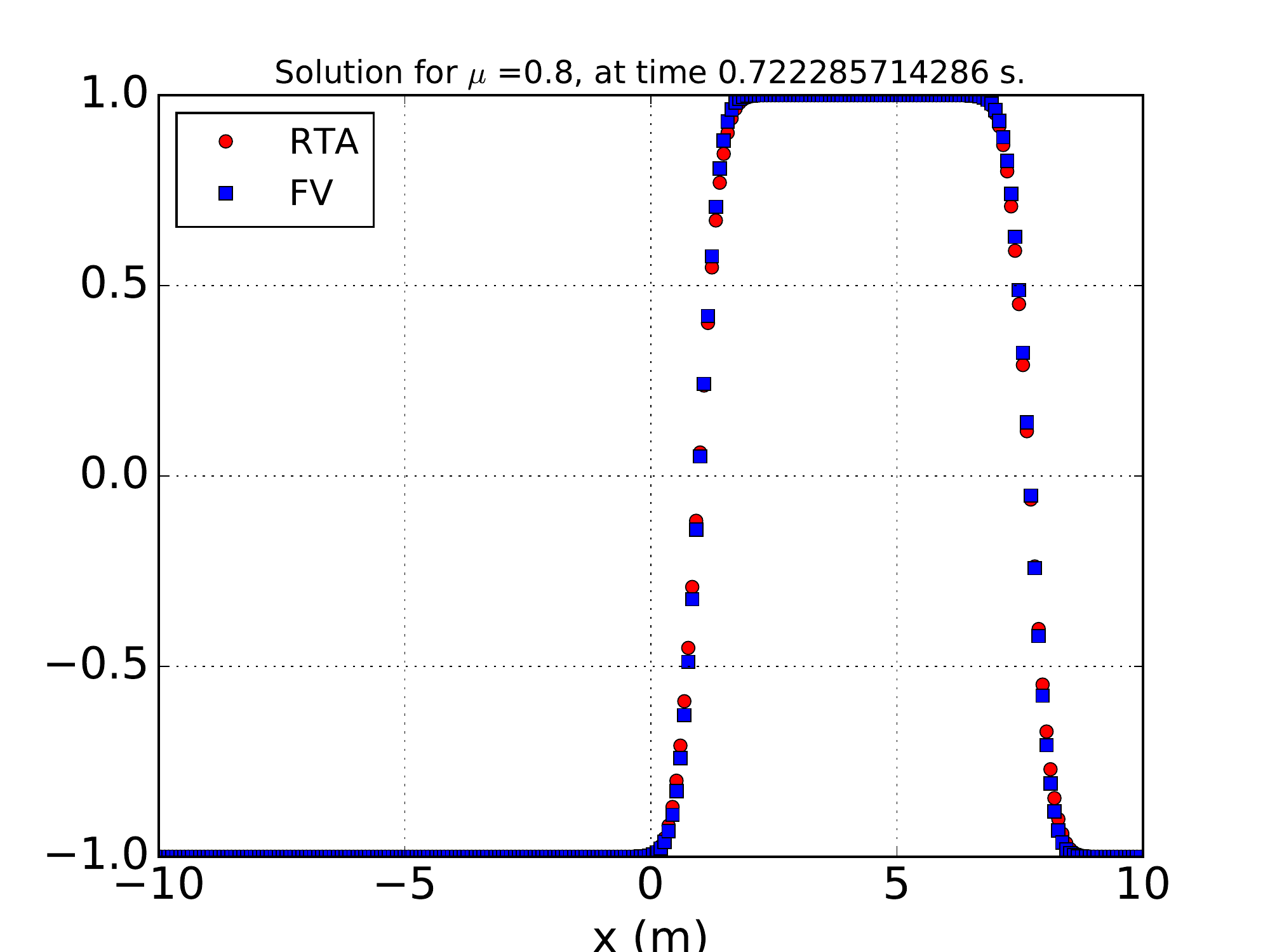} 
}
\subfigure[$t=2.16 \times 10^{-1}$ s.]{\label{transport1}
\includegraphics[width=0.33\textwidth]{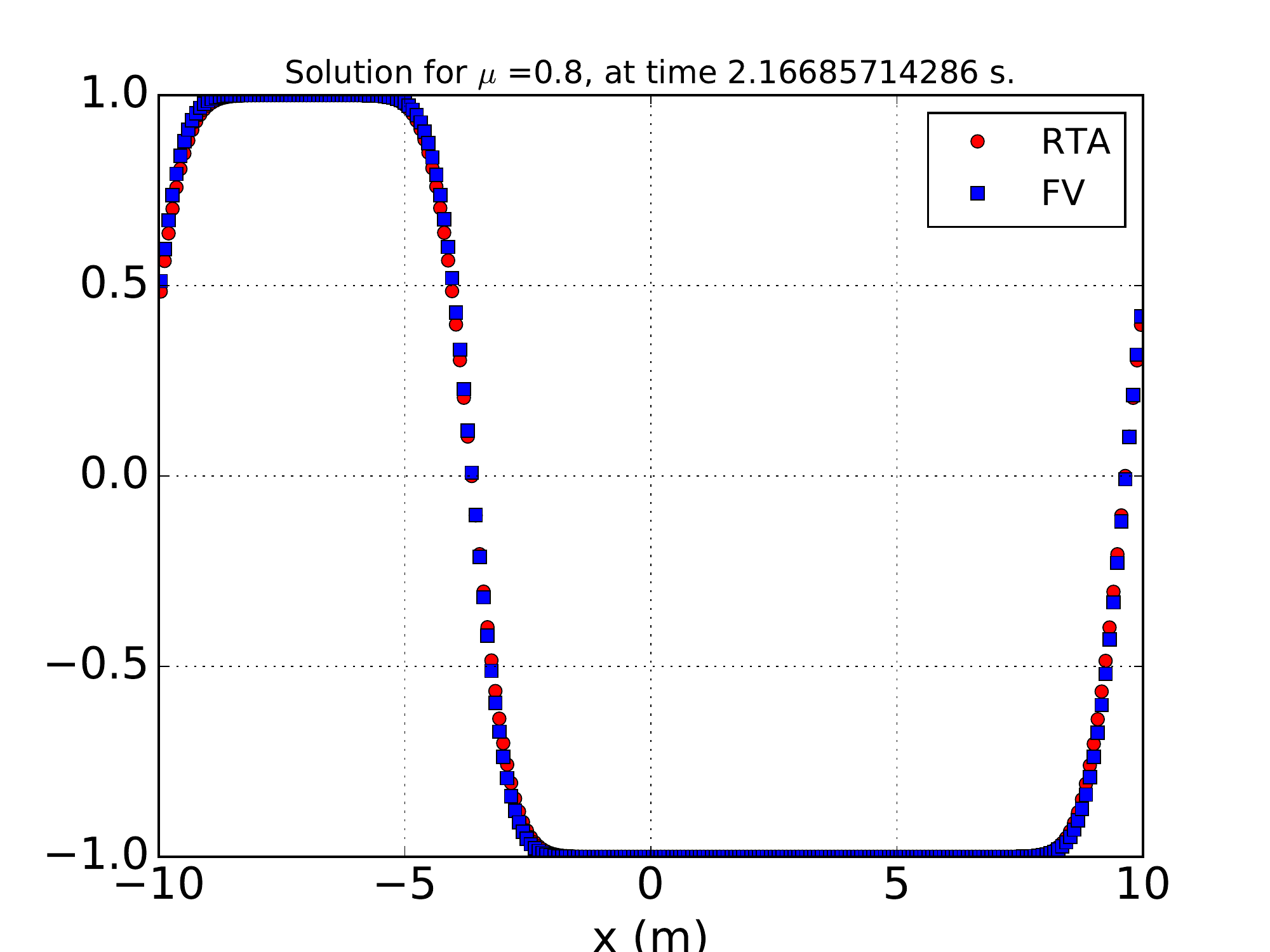} 
}
\subfigure[$t=8.14 \times 10^{-1}$ s.]{\label{transport1}
\includegraphics[width=0.33\textwidth]{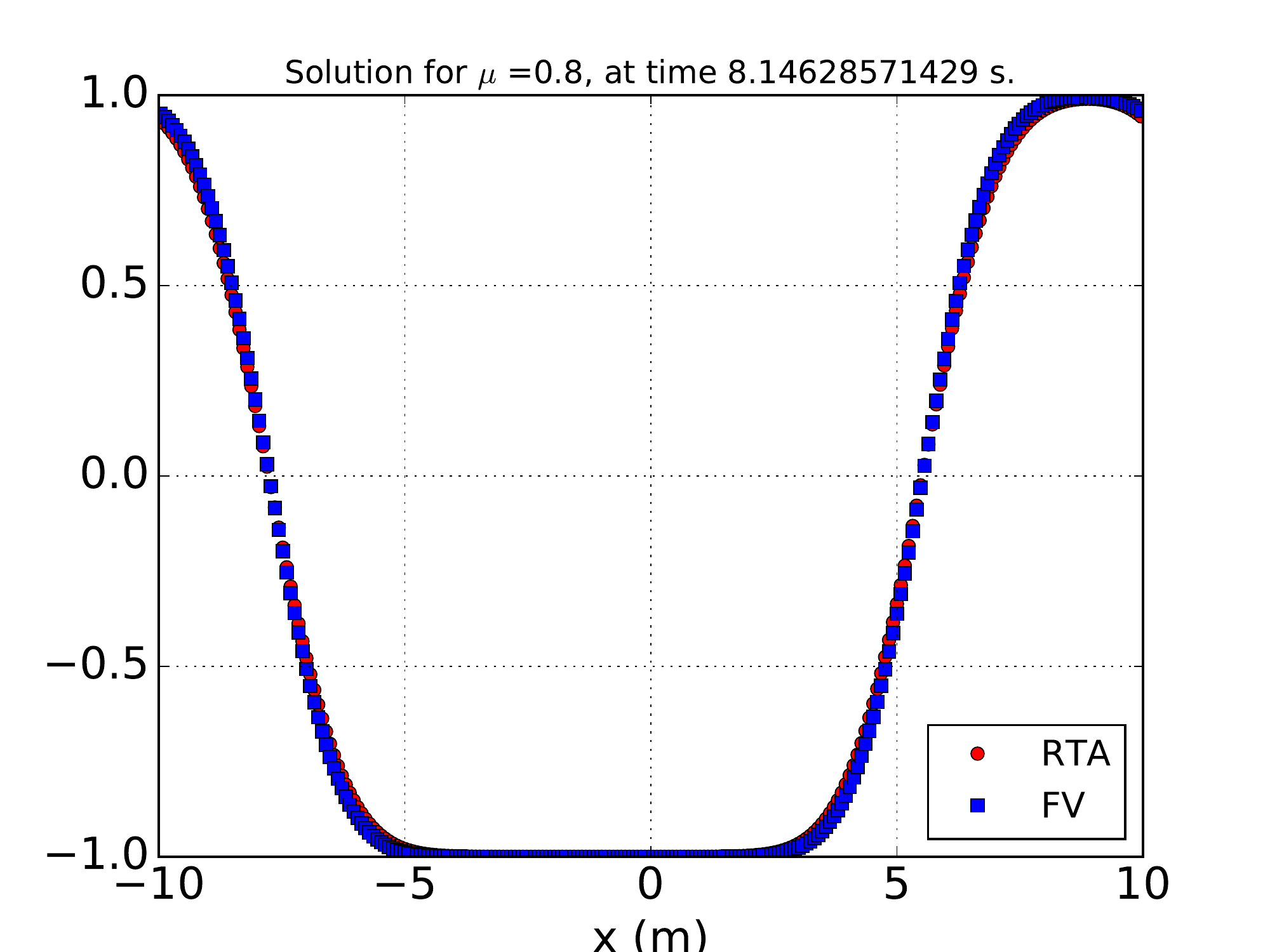} 
}
\caption{Comparison of RTA approximation computed from a snapshot obtained $\mu_i=0.4$ with the FV approximation computed for $\mu=0.8$ at different times. }
\label{transportSolution}
\end{figure}

To study the convergence of the RTA,  let us consider both absolute and relative error in discrete $L^1$ norm between $\Vphi_i^k(\mu)$ and $\Vu^k(\mu) $, at time $t^k$, given respectively by 
$$
e_a^k(\mu, \mu_i)= \Delta x \sum_{j=1}^N |\phi_{i,j}^k(\mu) - u_j^k(\mu) |  \text{ and }  e_r^k(\mu, \mu_i) =  \dfrac{ \sum_{j=1}^N |\phi_{i,j}^k(\mu) - u_j^k(\mu) |  }{\sum_{j=1}^N |u_{j}^k(\mu)|}.$$
Figure \ref{L1convCurves} shows both absolute and relative $L^1$ errors computed for different instances of the parameter $\mu$. Here, the reconstructed approximation, provided by the RTA method, is computed from a snapshot given for $\mu_i=0.65$. In accordance with theoretical results of Section \ref{sec:error}, the method converges since the approximation error decreases  with $\Delta x$. Moreover, the observed convergence rates are of the order of $1/2$ corresponding to those of the FV scheme for the approximation of discontinuous solutions, see \cite[Section 8.7]{Leveque2002}.  In that particular case, the projection error arising in Equation \eqref{eq:erreur_approximation} is negligible compared to the one of the FV scheme. The different constants and rates observed at the final time depend on $|\mu-\mu_i|$ as well as on the fractional part $\{k(\nu - \nu_i)\}$. This last observation motivates the interest of a selective procedure, as discussed in Remark  \ref{rk:dico}, for selecting the best snapshot in a pre-computed dictionary to minimize the approximation error. 

\begin{figure}[H]
\begin{center}
\includegraphics[width=0.8\textwidth]{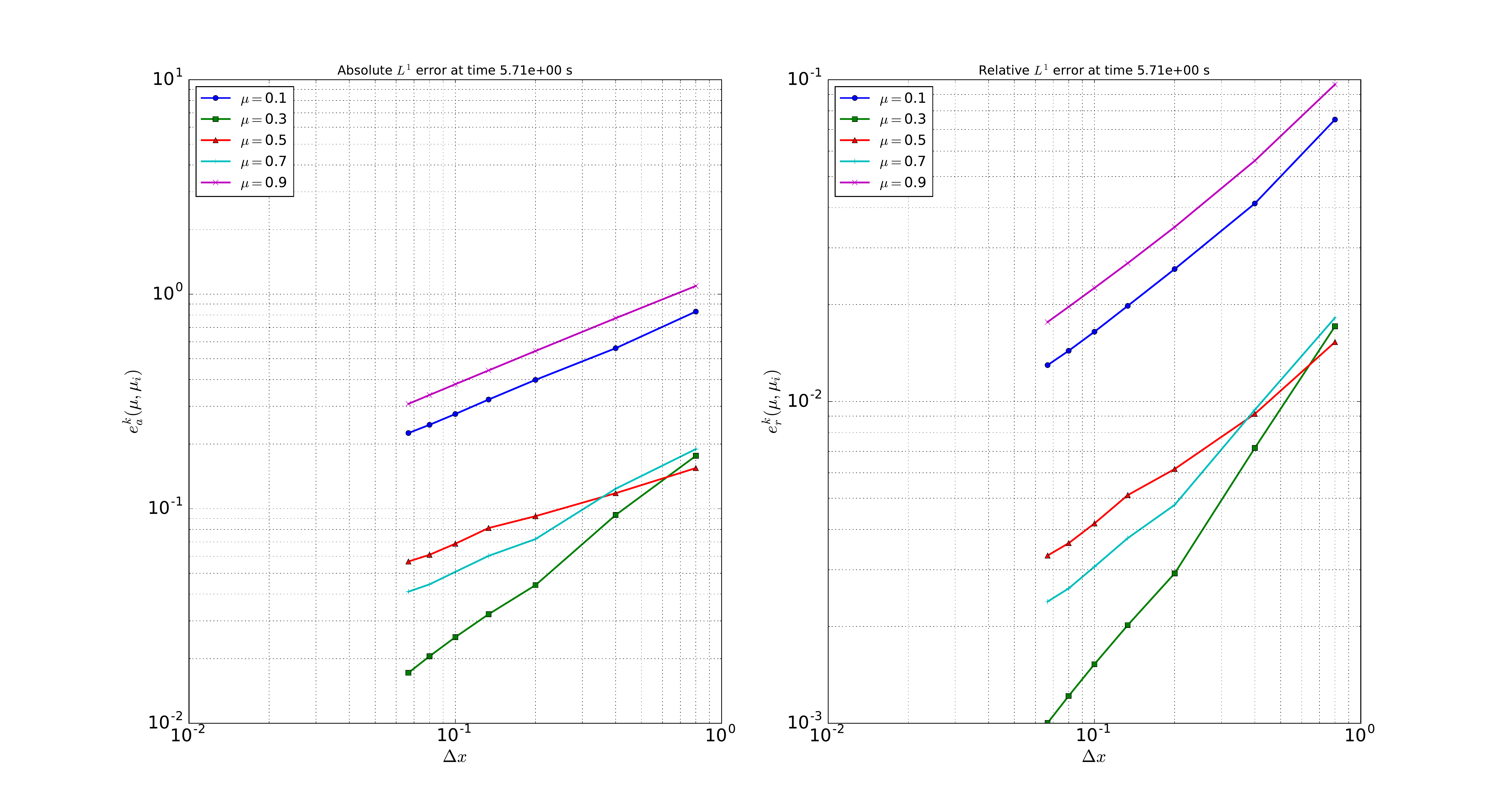} 
\end{center}
\caption{Convergence curves of $L^1$ absolute and relative errors computed for different values of the parameter $\mu$.}
\label{L1convCurves}
\end{figure}
Figures \ref{timeL1convCurvesMui} and \ref{timeL1convCurvesMeshes} show time evolution of the $L^1$ relative error superposed for different mesh sizes and  values of the parameter $\mu$. First, it can be observed on the two figures that the error remains globally of the same order of magnitude in time for a given mesh and parameter value, and do not increase exponentially. Indeed, contrary to time stepping algorithms, the proposed RTA method only requires the snapshots $\Vu^k(\mu_i)$ at current time $t^k$. In consequence, the approximation errors do not accumulate with time iteration. Second, it can be seen in Figure \ref{timeL1convCurvesMui} that error curves associated with the different instances of the parameter may cross in time. One possible explication of this observation is that the error is related to the fractional part $\{k(\nu-\nu_i)\}$ which depends on both $\mu, \mu_i$ and $k$. In particular, when the snapshots are aligned with the mesh at the instant $t^k$ one could expect a smallest error. 
Finally, as the mesh is refined, Figure \ref{timeL1convCurvesMeshes} shows that the error decreases monotonically as observed in Figure \ref{L1convCurves}.

\begin{figure}[H]
\begin{center}
\includegraphics[width=1.0\textwidth]{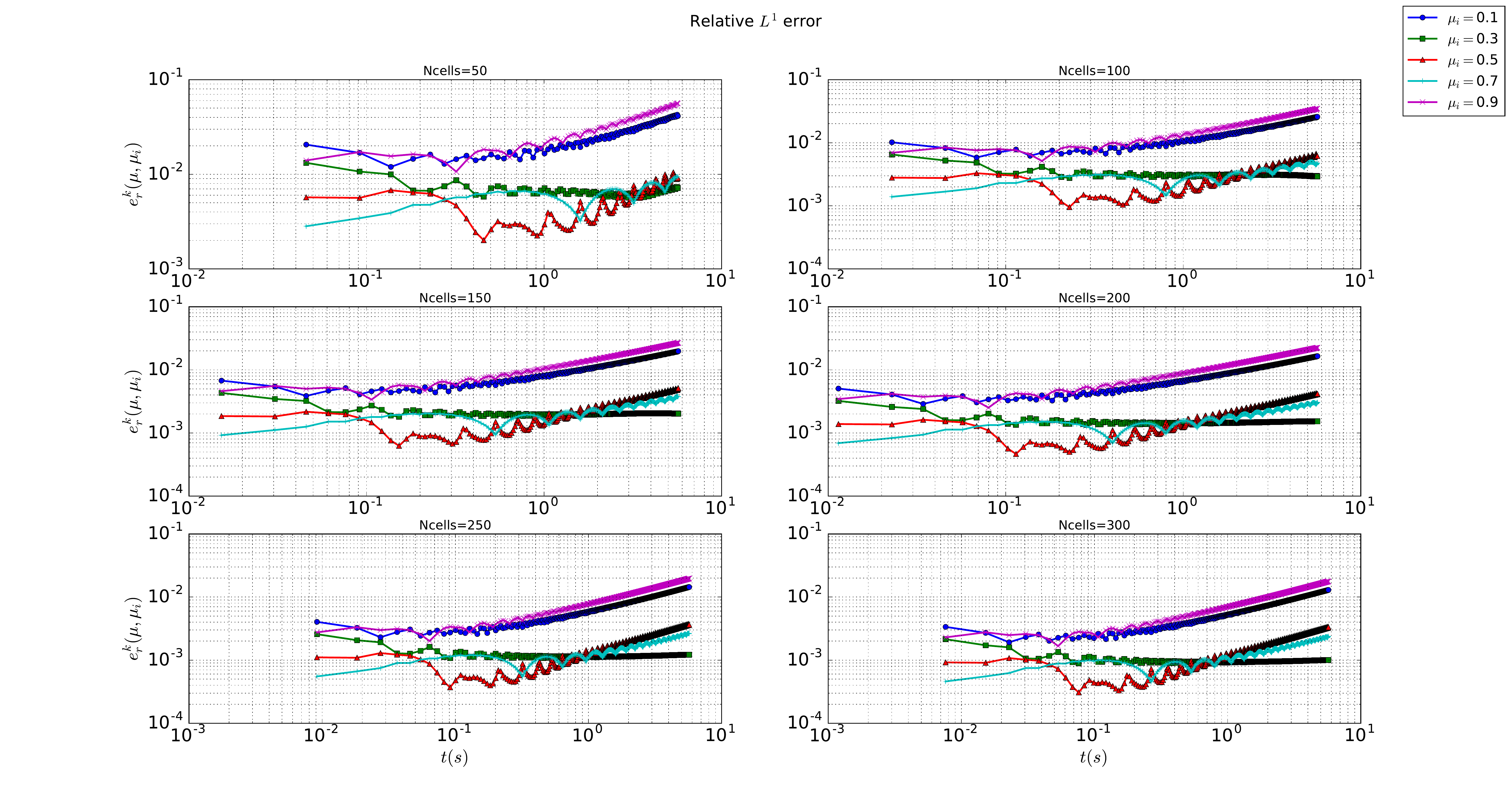} 
\end{center}
\caption{Time evolution of the $L^1$ relative error computed for different meshes and superposed for different values of the parameter $\mu$.}
\label{timeL1convCurvesMui}
\end{figure}
\begin{figure}[H]
\begin{center}
\includegraphics[width=1.0\textwidth]{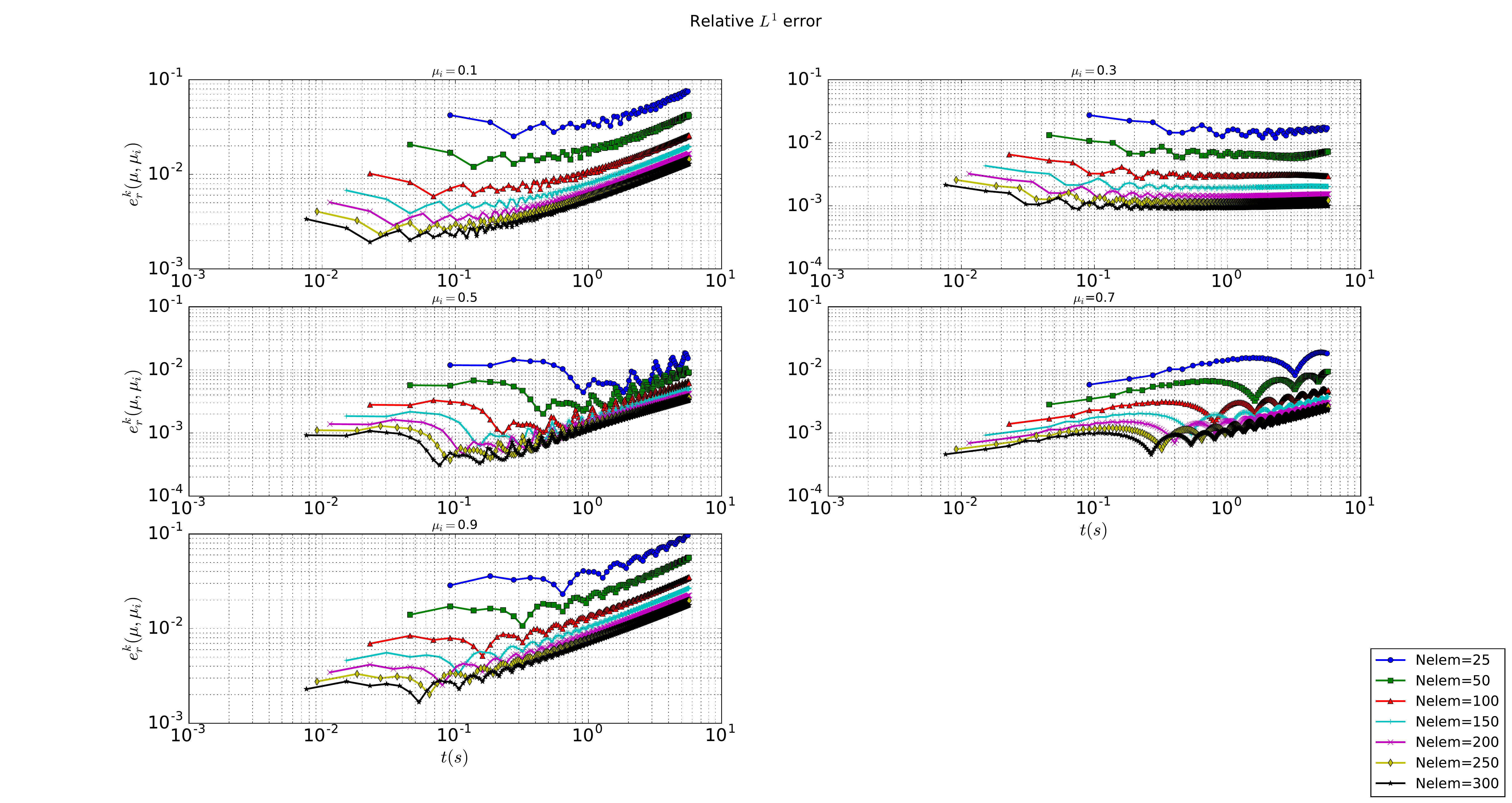} 
\end{center}
\caption{Time evolution of the $L^1$ relative error computed for different instances of the parameter $\mu$, superposed for different  meshes.}
\label{timeL1convCurvesMeshes}
\end{figure}

\subsection{System of linear conservation laws}

We conclude this section, with some numerical experiment on multiple transport parameter-dependent problem. 

\subsubsection{Linear elastodynamics}

Consider a bar, whose elastodynamic response in the isothermal and linearized geometrical framework is governed by the system of conservation laws \eqref{eq:syscons}, particularized with the following conserved variables $u : \Omega \times I \to \RR^2$ and flux $f : \RR^2 \to \RR^2$ given by
$$
u = \left(\begin{matrix}
 \sigma \\
 v
\end{matrix} \right) \quad \text{ and } \quad f(u) = \left( \begin{matrix}
 -E v \\
 - \frac{\sigma}{\rho}
\end{matrix}
 \right),
$$
where $\sigma : \Omega \times I \to \RR$ and $v : \Omega \times I \to \RR$ denote the Cauchy stress and the velocity respectively.  Moreover, $E$ and $\rho$ are two constants, corresponding to the Young modulus and the mass density respectively. This system of equations is also supplemented with periodic boundary conditions and given initial condition $u^0$. The above system can be set into a characteristic form \cite{Toro2013} which is given through the decoupled system of two equations
\begin{equation}
\label{decouple}
\begin{array}{rcl}
\dfrac{\partial w_1}{\partial t} + c \dfrac{\partial w_1}{\partial x} &=& 0, \\[0.2cm]
\dfrac{\partial w_2}{\partial t} -c \dfrac{\partial w_2}{\partial x} &=&0,
\end{array}
 \qquad \text{ on } \Omega \times I.
\end{equation}
where $c=\sqrt{E/\rho}$ is the elastic celerity of the bar. The vector $w = (w_1,w_2)^T$  contains the characteristic variables $w_l : \Omega \times I \to \RR$ which satisfy $u = R w$ with 
 $R \in \RR^{2 \times 2}$ the matrix of right eigenvectors of the jacobian matrix of the flux $f$ defined as
\begin{equation}\label{eigenvectors} 
R = \left(\begin{matrix}
\rho c & - \rho c \\
1 & 1 
\end{matrix}
\right).
\end{equation}
In what follows, the elastic celerity $c(\mu)=\sqrt{E(\mu)/\rho}$ depends on the parameter $\mu \in \Pc$, through the Young modulus $E(\mu)$. 
We consider the problem of computing an approximation of the FV solution of this elastodynamic problem for a target value $\mu$. To that goal, the RTA method is considered for approximating the FV solution of each independent characteristic equation. First, at each time $t^k$, provided snapshots $\mathbf{u}^k(\mu_i) \in \RR^N \times \RR^2$, the corresponding FV approximations of characteristic quantities $\mathbf{w}^k(\mu_i)  \in \RR^N \times \RR^2$ are  computed as
\begin{equation}
\mathbf{w}^k(\mu_i) = \mathbf{u}^k(\mu_i) R^{-T}.
\end{equation}
Next, the reconstructed approximations computed by the RTA method are 
\begin{equation}
{\Vphi}_{i,l}^k (\mu)= {\cal K}(k (\nu-\nu_i)) \boldsymbol{w}_{l}^k(\mu_i), \quad l=1,2
\end{equation}
where $\boldsymbol{w}_{l}^k(\mu_i) \in \RR^N$ are the columns of $\mathbf{w}^k(\mu_i)$. Here ${\Vphi}_{l,i}^k (\mu) \in \RR^N$ stands for the RTA reconstruction from $\boldsymbol{w}_{l}^k(\mu_i)$. Finally, the numerical approximation of the conservative variables $u$ for the value $\mu$ are obtained by recombining the reconstructed approximations of characteristic variables ${\Vphi}_{i,l}^k (\mu), l=1,2$, using the matrix $R$
\begin{equation}
\label{eq:rebuilt}
\VPhi_{i,l}^k (\mu) = \sum_{m=1}^2 R_{lm} {\Vphi}_{i,m}^k (\mu) 
\end{equation}
where $\VPhi_{i,l}^k (\mu) \in \RR^N$ consists of the approximations for the translated stresses ($l=1$) and for the velocity ($l=2$) respectively.

\subsubsection{Results}

For the following simulations, the Young modulus $E(\mu)$ depends on $\mu\in [0,1]$ through $E(\mu)=c_0 \mu + c_1$, with $c_0,c_1$ being two constants set at $19\times 10^{10}$ Pa and $10^{11}$ Pa respectively. The mass density is set at $7800$ kg.m$^{-3}$. The bar is initially free of any stresses $\sigma^0 =0$. Riemann-type initial conditions are prescribed on the velocity field so that  $v^0=1$m.s$^{-1}$ in the first half of the medium $x \in [-L,0)$, and $v^0=0$ in the second half $x \in (0,L]$, so that compression first occurs at the middle of the computational domain.

Figure \ref{w_elastoSolution} shows some superposed plots of the approximations of the characteristic variables  $w_l(\mu) = \frac 12 (\mp \frac{\sigma (\mu) }{\rho c(\mu) } + v(\mu))$, $l=1,2$ computed with the RTA  method and the FV scheme for $\mu=0.8$, as well as the snapshot consisting of a FV approximation obtained with $\mu=0.05$. The mesh consists of $250$ cells, and the maximum $CFL$ number is set at $0.8$ for $\mu=1$. It can be observed a very good agreement on both characteristic variables between RTA and FV approximation for the instance value $\mu=0.8$. As the time increases, the space shift between the snapshot and the RTA solution also significantly increases, showing that the CFL-like condition $\Delta a(\mu,\mu_i)t^k \leq \Delta x$ of remark \ref{CFLlike} is not satisfied.\\

\begin{figure}[H]
\begin{center}
\subfigure[$t=1.29 \times 10^{-3}$ s.]{\label{w_elasto1}
\includegraphics[height=0.3\textheight]{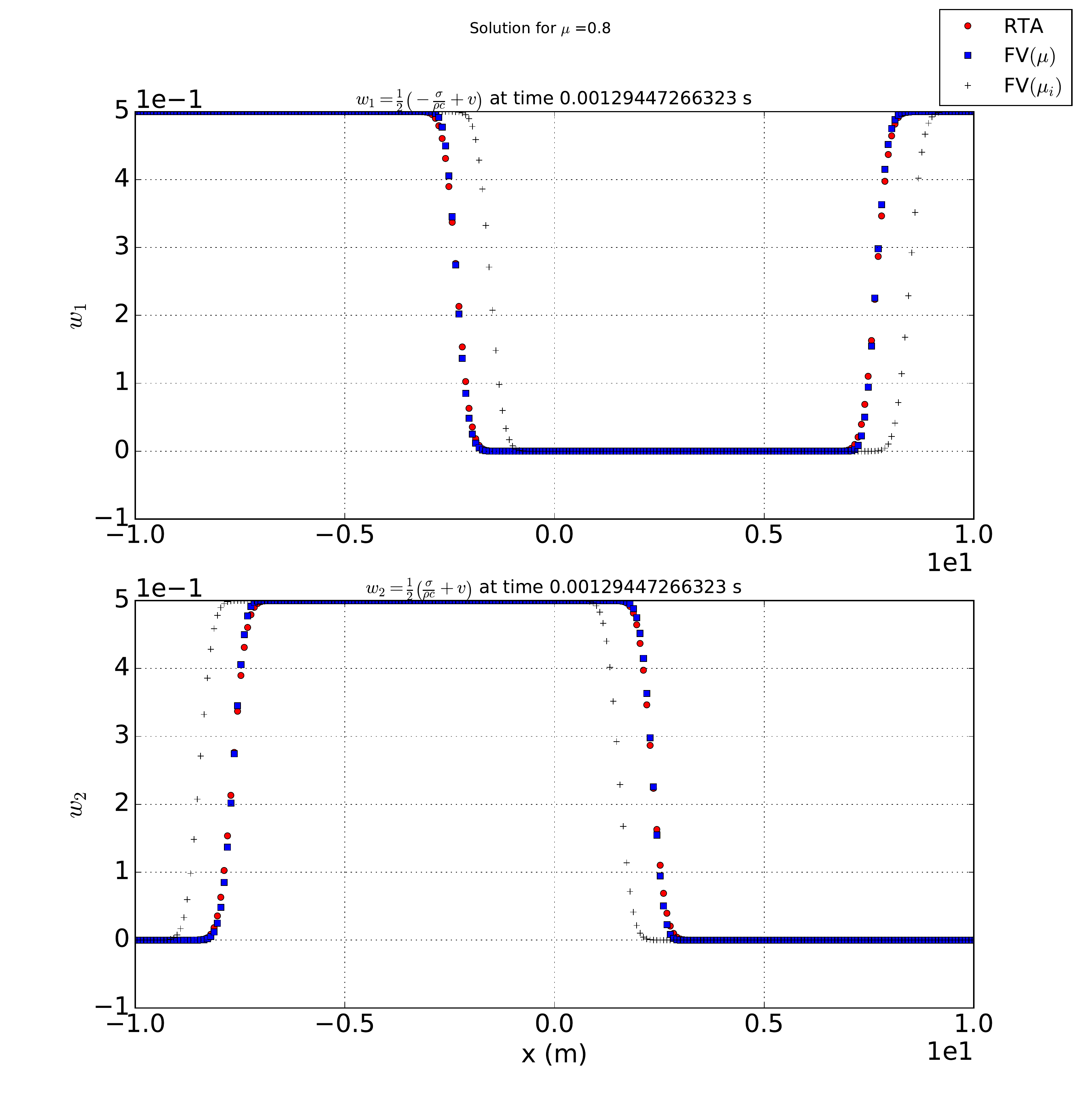} 
}
\subfigure[$t=3.88 \times 10^{-3}$ s.]{\label{w_elasto2}
\includegraphics[width=0.3\textheight]{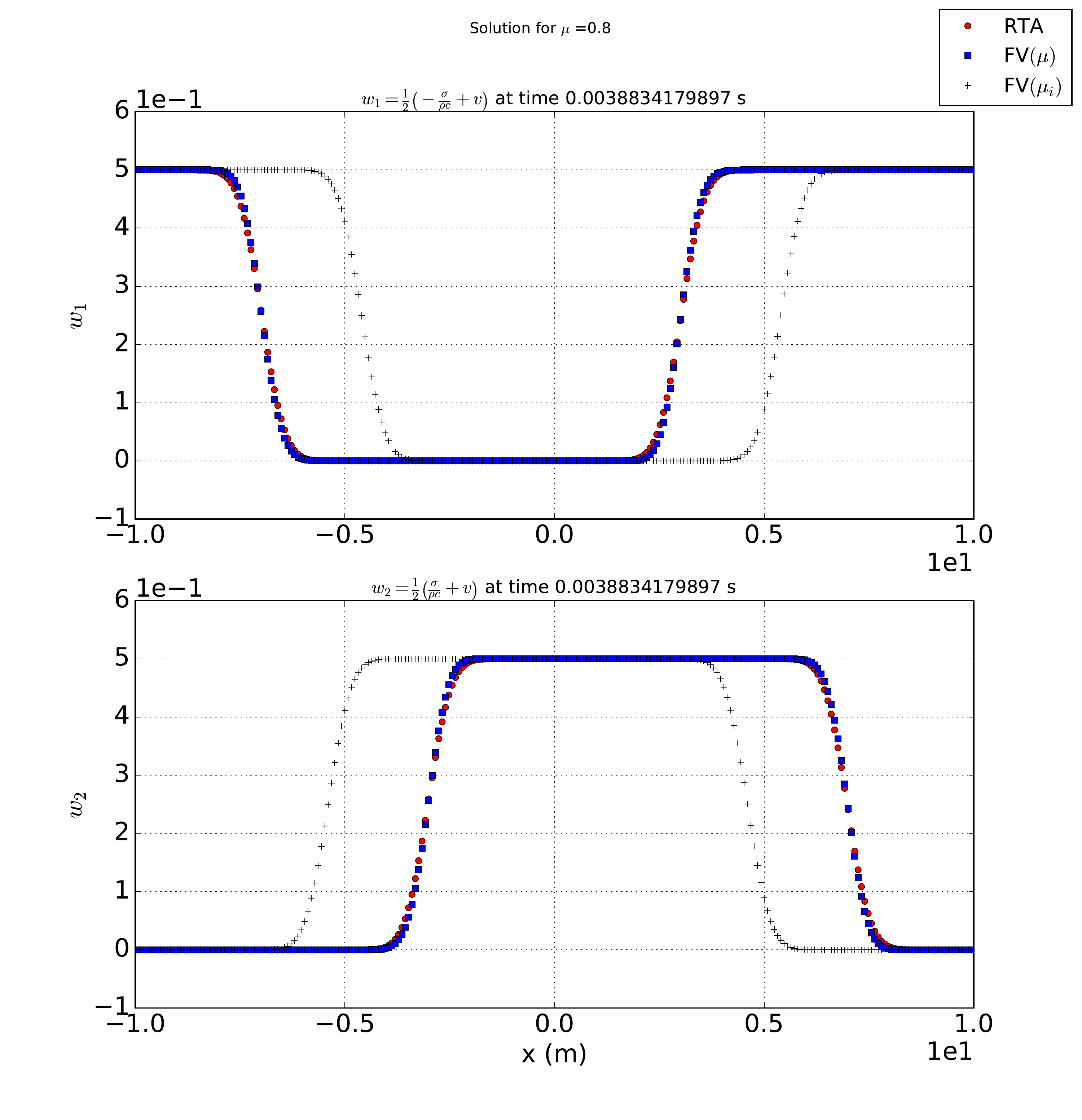} 
} \\
\subfigure[$t=9.06 \times 10^{-3}$ s.]{\label{w_elasto3}
\includegraphics[width=0.3\textheight]{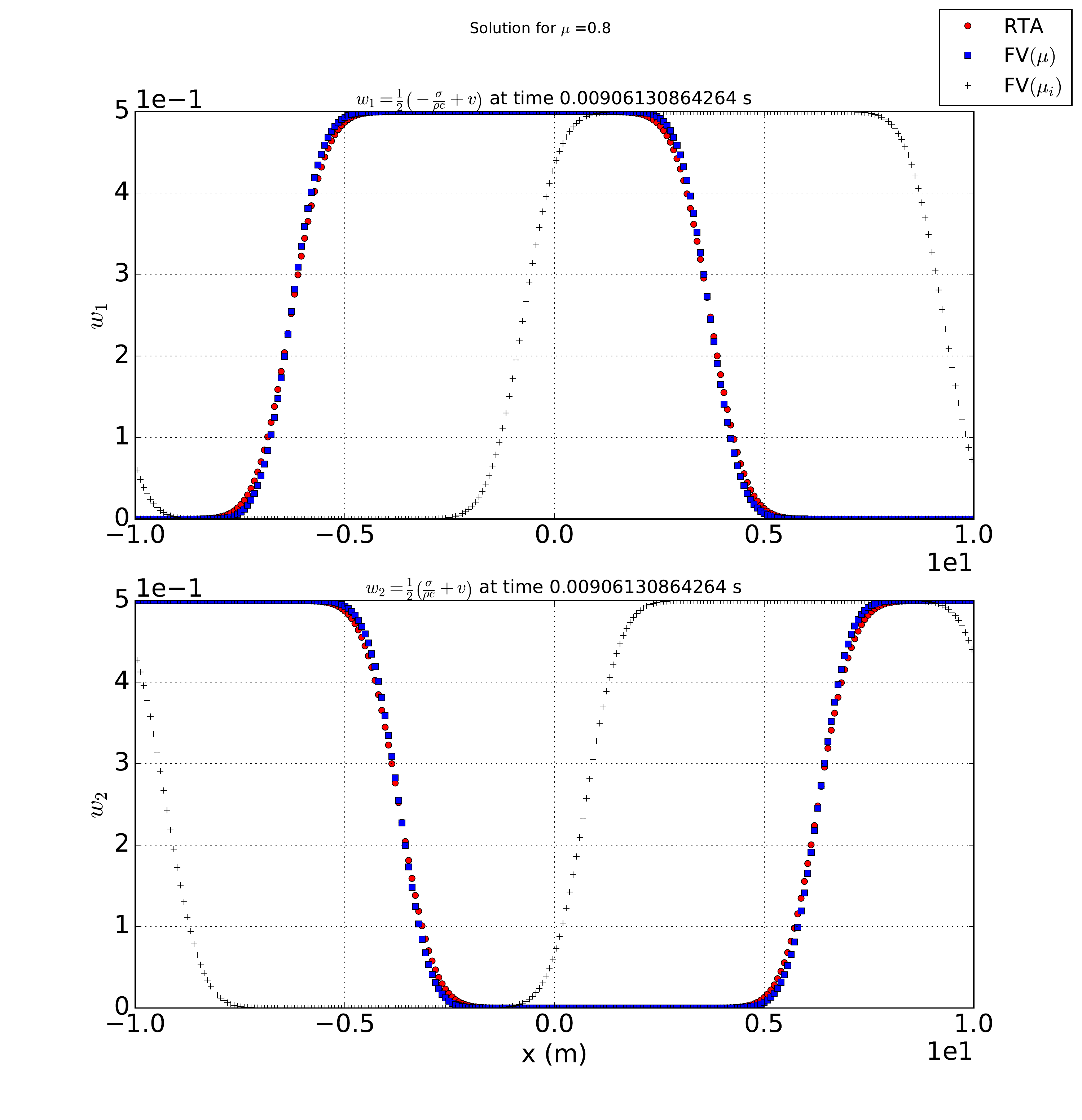} 
}
\subfigure[$t=1.29 \times 10^{-2}$ s.]{\label{w_elasto4}
\includegraphics[width=0.3\textheight]{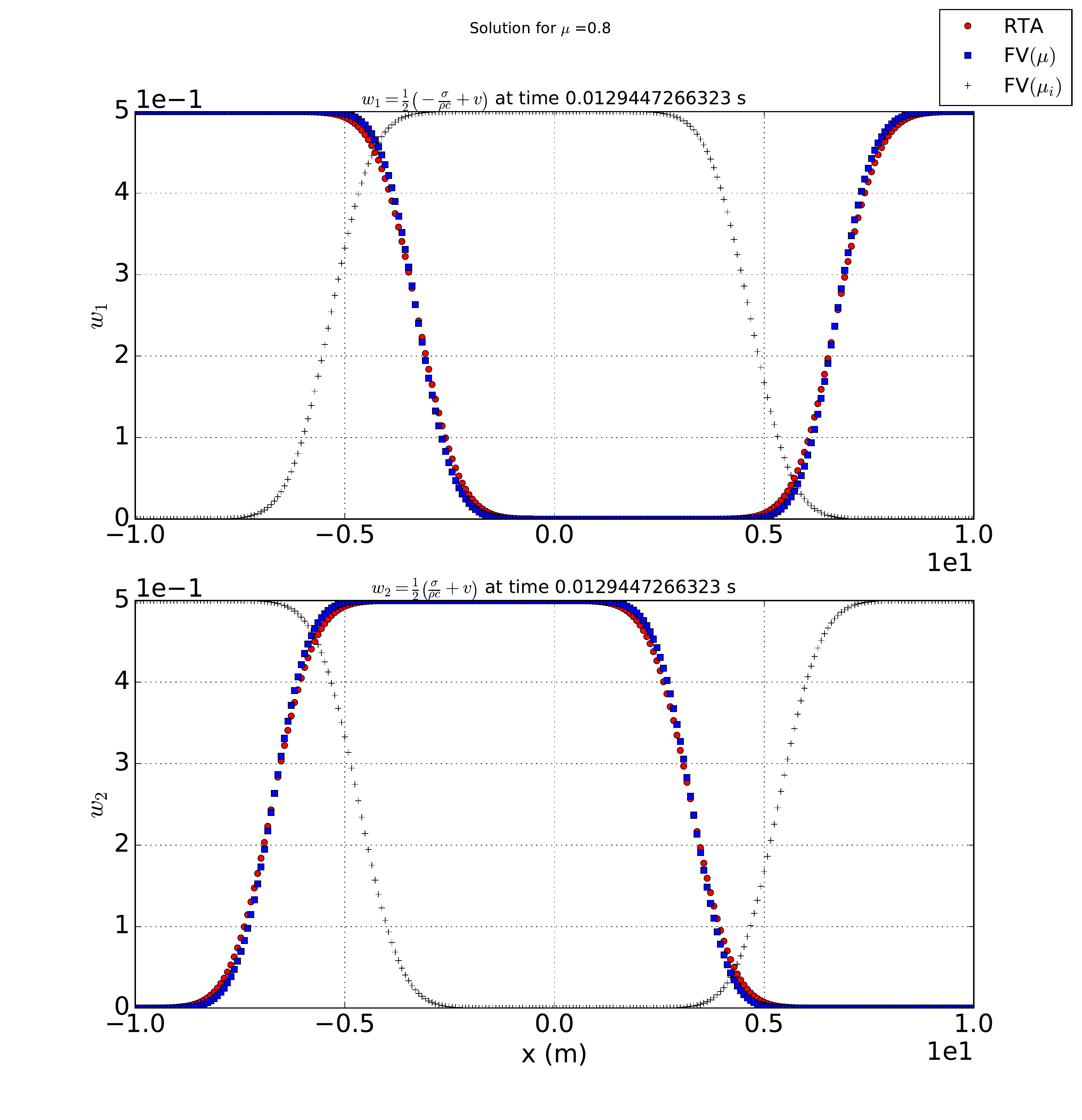} 
}
\end{center}
\caption{Comparison of the characteristic variables $w_l(\mu)$, $l=1,2$ : approximations computed with RTA, from a snapshot obtained for $\mu=0.05$, and FV for $\mu=0.8$ at different times. }
\label{w_elastoSolution}
\end{figure}

Finally, Figure \ref{elastoSolution} shows the rebuilt conserved variables obtained with RTA and FV approximations computed for the parameter value $\mu=0.8$ at different times using Equation \eqref{eq:rebuilt}. The two solutions are again in very good agreement for any times. 

\begin{figure}[H]
\begin{center}
\subfigure[$t=1.29 \times 10^{-3}$ s.]{\label{elasto1}
\includegraphics[height=0.3\textheight]{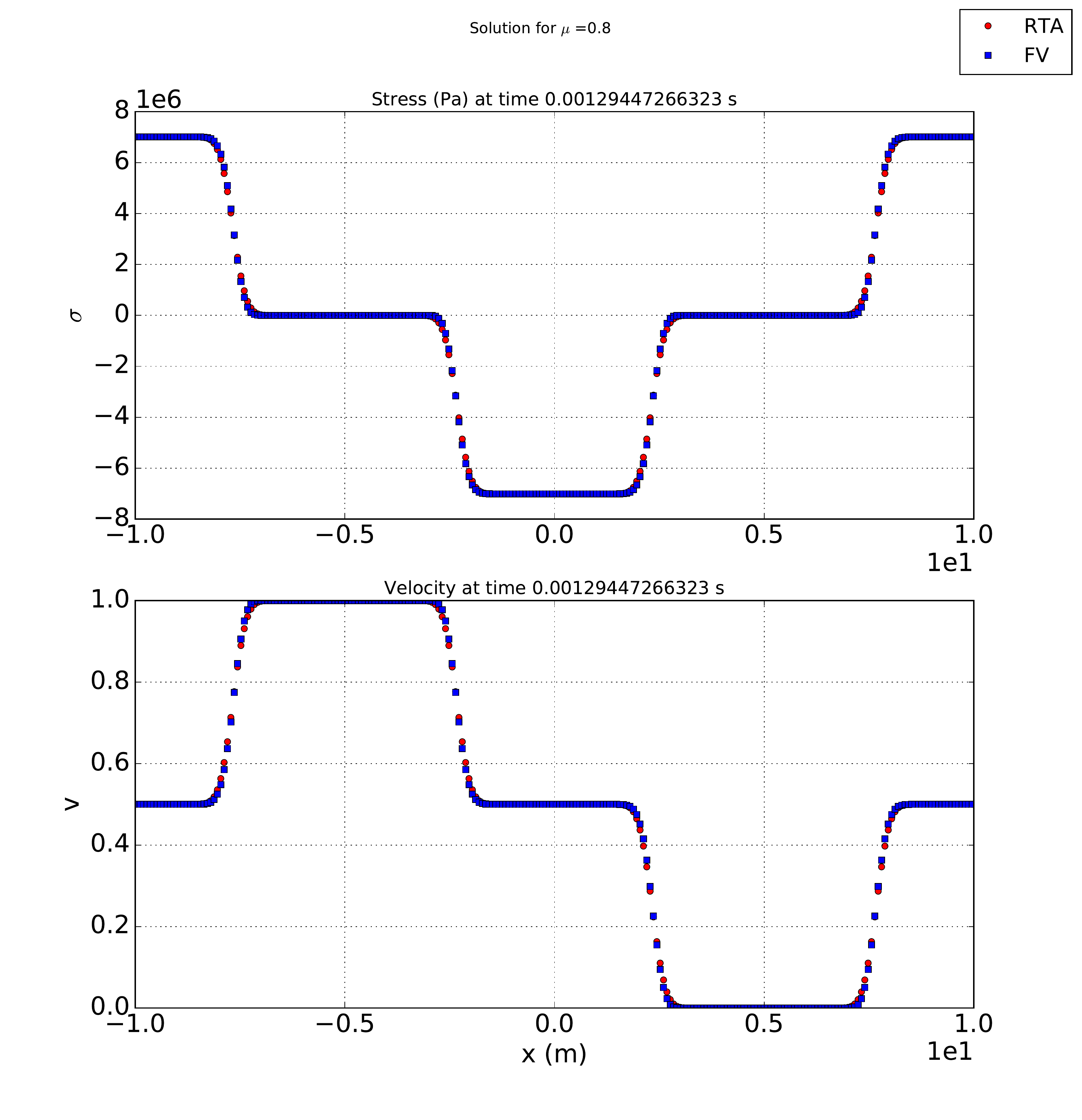} 
}
\subfigure[$t=3.88 \times 10^{-3}$ s.]{\label{elasto2}
\includegraphics[width=0.3\textheight]{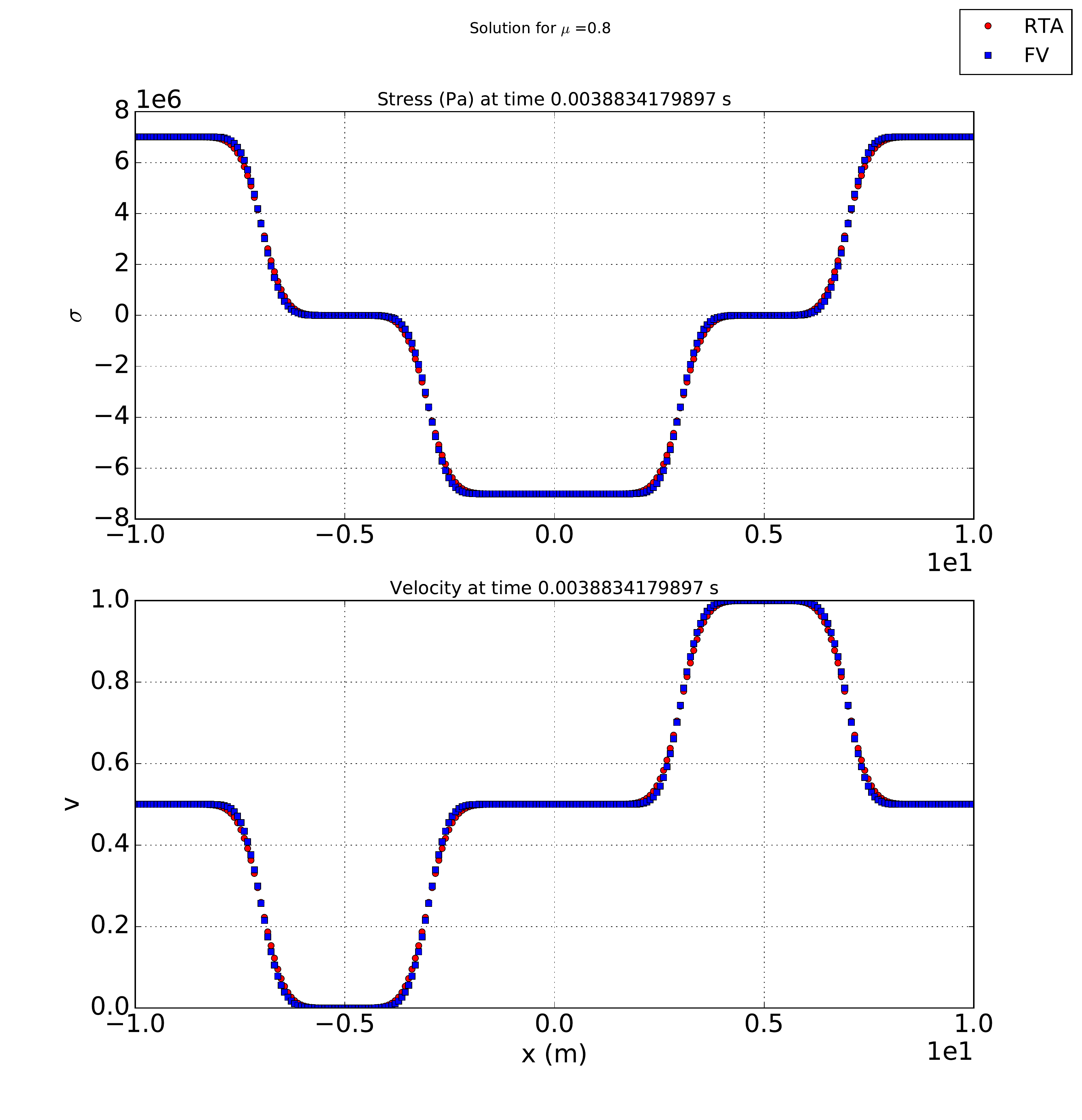} 
} \\
\subfigure[$t=9.06 \times 10^{-3}$ s.]{\label{elasto3}
\includegraphics[width=0.3\textheight]{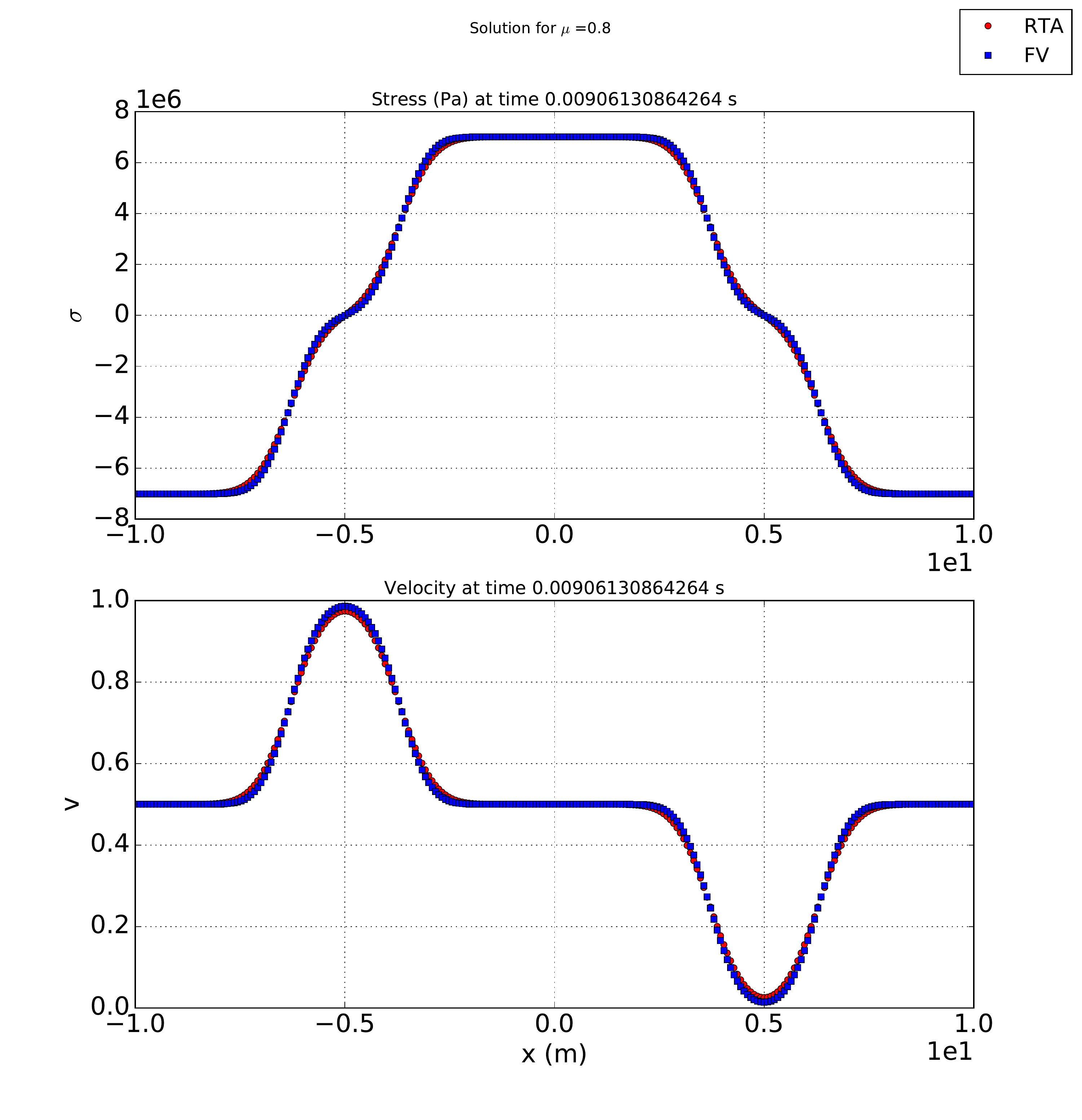} 
}
\subfigure[$t=1.29 \times 10^{-2}$ s.]{\label{elasto4}
\includegraphics[width=0.3\textheight]{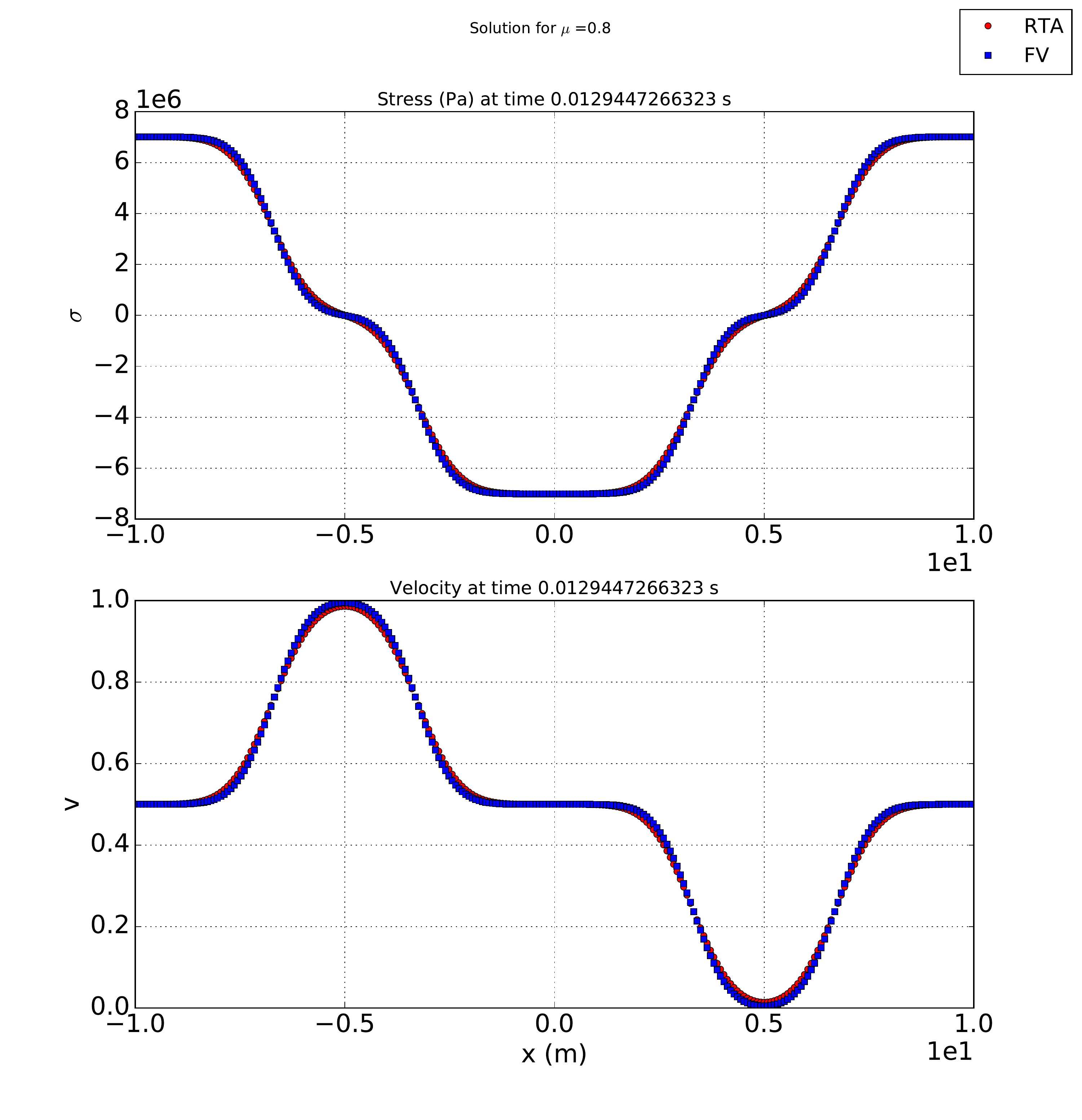} 
}
\end{center}
\caption{Comparison of stress and velocity fields : approximations computed with RTA, from a snapshot obtained for $\mu=0.05$, and FV for $\mu=0.8$ at different times. }
\label{elastoSolution}
\end{figure}

%% file: conclusion.tex
\section{Conclusion}\label{conclusion}


A numerical method for the approximation of finite volume solutions of parameter-dependent linear hyperbolic conservations laws has been proposed in this work. It allows to reconstruct an approximation of the finite volume solution for any parameter value, from only one single snapshot obtained for a given instance of the parameter with a first order upwind finite volume solver.  In the case of the linear parameter dependent transport equation, the approximation is built with the simple Reconstruct-Translate-Average algorithm that allows to reconstruction a piecewise constant function at fixed time $t^k$ from translated snapshot using the known characteristics of the equations. The provided approximation can be interpreted as a rank one low rank approximation of the finite volume solution, up to discretization error at any time. 
In practice, an offline-online implementation allows to efficiently compute the approximation for any time $t^k$ without any time stepping nor projection step. The provided approximation is proven to be total variation bounded. Moreover, a bound of the  approximation error has been derived proving that the method converges as the mesh used for the upwind scheme is refined. Finally, the error is only local in time and does not   not accumulate with time iterations. 

In this paper, the RTA algorithm has been detailed for periodic  parameter transport equation in only one dimension. Within the same framework, it can be naturally extended to  higher dimensional problems with more general boundary conditions than periodic ones.  Moreover, higher order approximations are  possible during the reconstruction step (Step 1 of Algorithm \ref{alg:RTA}).\\
Regarding to the problem of approximation, we have proposed here a strategy to compute efficiently a robust approximation of finite volume scheme from snapshots.  As discussed in Remark \ref{rk:dico}, designing computable and sharp \textit{a posteriori} error estimate could allow to design a selective procedure of snapshots lying in a dictionary, in order to improve the error of the reconstructed approximation. More generally, the proposed approach represents a first step toward  efficient non linear ROM strategies, in particular  to design dynamical RB method with adapted local basis in finite volume framework. In the lines of \cite{Cagniart2017,Rim2019}, using RTA approach together with efficient strategies to compute both transformations and reduced basis  for ROM of  finite volume solution of more general hyperbolic conservation laws will be the object of future work. 

\paragraph{Acknowledgment}
 This work has been partially funded by the CNRS Energy unit (Cellule Energie) through the project DROME.